\documentclass[12pt]{amsart}
\usepackage{amsmath}
\usepackage{amssymb}
\usepackage{amsthm}
\usepackage[usenames]{color}
cm \hoffset -1.5 true cm \marginparwidth= 2 true cm

\newtheorem{thm}{Theorem}[section]
\newtheorem{cor}[thm]{Corollary}
\newtheorem{lem}[thm]{Lemma}
\newtheorem{prop}[thm]{Proposition}
\theoremstyle{definition}

\theoremstyle{remark}
\newtheorem{rem}[thm]{Remark}
\numberwithin{equation}{section}

\newcommand{\supp}{\text{supp }}
\newcommand{\dist}{\text{dist }}
\newcommand{\inp}[2]{\langle #1,#2 \rangle}

\newcommand{\schr}{e^{it\Delta}}
\newcommand {\mix}[2]{L^{#1}_tL^{#2}_x}
\newcommand{\mixednorm}{\|_{\mix{q/2}{r/2}}}

\newcommand{\amixednorm}{\|_{\mix{\tilde q'}{\tilde r'}}}
\newcommand{\amixed}{\mix{\tilde q'}{\tilde r'}}
\newcommand{\amixedI}{\mix{\tilde q'}{\tilde r'}(J_T\times \bbr^n)}
\newcommand {\Del} {\Delta}

\newcommand{\hilnorm}[1]{\|_{H^{#1}}}

\newcommand{\hhilnorm}[1]{\|_{{\overset{\centerdot}{H}}{}^{#1}}}
\newcommand{\La}{\Lambda}
\newcommand{\al}{\alpha}

\newcommand{\xsp}{\mathbf X^{s,b}}
\newcommand{\xspnorm}{\|_{\mathbf X^{s,b}}}

\newcommand{\xsnorm}[2]{\|_{\mathbf X^{#1,#2}}}
\newcommand{\na}{\nabla}
\newcommand{\la}{\lambda}

\newcommand{\bbr}{\mathbb R}
\newcommand{\tfg}{|\nabla|^{2-n} (\schr f\schr g)\schr h}

\newcommand{\calh}{\mathcal H (f, g, h)}
\newcommand{\calhd}{\mathcal H (f, g, \nabla h)}
\newcommand{\caldh}{\mathcal H (\nabla f, g, h)}

\newcommand{\tfgdij}[3]{\mathcal H (f_{N_{#1}}, g_{N_{#2}}, \nabla h_{N_{#3}})}

\newcommand{\tddfgij}[3]{\mathcal H (\nabla f_{N_{#1}},  g_{N_{#2}},  h_{N_{#3})}}
\newcommand{\tddfgijk}{\tddfgij{1}{2}{3}}
\newcommand{\tfgdijk}{\tfgdij{1}{2}{3}}

\newcommand{\normfg}{\|f\|_{L^2}\|g\|_{L^2}\|h\|_{L^2}}

\newcommand{\hnormfg}{\|f\|_{\overset{\centerdot} {H}{}^{s_1}}\|g\|_{\overset{\centerdot} {H}{}^{s_2}}\|h\|_{\overset{\centerdot} {H}{}^{s_3}}}

\newcommand{\xsnormuv}{\|u\xsnorm{s_1}{b}\|v\xsnorm{s_2}{b}\|w\xsnorm{s_3}{b}}

\newcommand \pair  {\tx{\frac {4(n-1)}{n}}{\frac {2(n-1)}{n-2}}}

\newcommand{\sfa}[1]{N_{#1}^{s_1}\|f_{N_{#1}}\|_{L^2}}
\newcommand{\sga}[1]{N_{#1}^{s_2}\|g_{N_{#1}}\|_{L^2}}
\newcommand{\sha}[1]{N_{#1}^{s_3}\|h_{N_{#1}}\|_{L^2}}

\newcommand{\ssum}{\sum s_i=1}
\newcommand {\gain}{\left(\frac{N_4}{N_2} \right)^{\frac 12}}
\newcommand{\sss}{s_1,s_2,s_3}
\newcommand {\mi}[1]{L^{#1}_{\xi_1,\eta}}
\newcommand \tx[2] {L^{#1}_tL^{#2}_x}

\newcommand {\ep} {\epsilon}
\newcommand {\lnabla} {\langle \nabla \rangle}

\newcommand{\fd}[1]{|\nabla|^{#1}}
\newcommand{\pa}{\partial}
\begin{document}
\title[Mixed norm estimates and their applications]
{Mixed norm estimates of Schr\"{o}dinger waves and their applications}
\author{Myeongju Chae}
\address{Department of Applied Mathematics, Hankyong National University, Ansong 456-749, Republic of Korea}
\email{mchae@hknu.ac.kr}
\author{Yonggeun Cho}
\address{Department of Mathematics, and Institute of Pure and Applied Mathematics, Chonbuk National University, Jeonju 561-756, Republic of Korea}
\email{changocho@chonbuk.ac.kr}
\author{Sanghyuk Lee}
\address{Department of Mathematical Sciences, Seoul National University, Seoul 151-747, Republic of Korea}
\email{shklee@snu.ac.kr}
\subjclass[2000]{42B25, 35Q40, 35Q55}
\keywords{mixed norm estimates, interactive Schr\"{o}dinger waves, mass critical nonlinearity, smoothing property, global well-posedness}
\thanks{Y. Cho is supported in part by the Korea Research Foundation Grant KRF-2008-313-C00065. S. Lee is supported in part by the grant
KOSEF-2007-8-1220.}
\begin{abstract}
In this paper we establish mixed norm estimates of interactive
Schr\"{o}dinger waves and apply them to study  smoothing
properties and global well-posedness of the nonlinear
Schr\"{o}dinger equations with mass critical nonlinearity.
\end{abstract}
\maketitle

\section{Introduction}
The Sctrichartz estimate shows the dispersive nature of Schr\"{o}dinger waves, which can be formulated via mixed norm (\cite{s, kt}). More precisely, for admissible $(q,r)$
$$\|\schr f\|_{\mix{q}{r}}\lesssim \|f\|_{L^2}.$$
Here a pair $(q, r)$ is said to be  {\it admissible} if it satisfies
$\frac2q=n(\frac 12 - \frac1r)$, $q,r\ge 2$ with exception
$(q,r)=(2,\infty)$ when $n=2$ and  $e^{it\Delta}$ denotes the  free
propagator of Schr\"{o}dinger equation.

Due to scaling, the frequency localization via Littewood-Paley decomposition does not give any improvement to the aforementioned Strichartz estimates.
 However, it was observed by Bourgain \cite{bo} that by considering  low and high frequency  interactions  of two Schr\"{o}dinger waves, namely bilinear control of $\schr f\schr g$,  it is possible to obtain  a refinement of Strichartz estimate in $L_{t,\,x}^2(\bbr \times \bbr^2)$ (note that $(4,4)$ is an admissible pair when $n=2$). In \cite{kv} Keranni and Vargas recently extended Bourgian's reults to higher dimensions by showing that a sharp $L_{t,\, x}^{(n+2)/n}(\bbr \times \bbr^n), n \ge 1$ estimate holds for the interactive Schr\"{o}dinger waves.

 Our first result is that such refinements of Strichartz estimates are also  valid in the mixed norm setting for $n \ge 3$. Actually it gives stronger interactive estimate which is stated as follows:
\begin{thm}\label{hilbert} Let $n\ge 2$.  Let $(q,r)$ satisfy that $2/q=n(1/2 - 1/r)$,  $2 < r < 4$, and $q>2$. Then for $|s| <1-2/r$,
\begin{align}\label{bilinear}
\|\schr f\schr  g\|_{\mix {q/2} {r/2}} \lesssim \|f\hhilnorm{s}\|g\hhilnorm{-s}.
\end{align}
\end{thm}
\noindent The estimate trivially holds for $(r,s)=(2,0)$ by Plancherel's theorem and when $n=2$ it was actually obtained in \cite{kv} including $(q,r)=(4,4)$. This estimate obviously has a scaling structure in $L^2$ space so that the estimate is invariant along the admissible $(q,r)$.
The above estimate makes  it possible to move a certain amount of derivative on  one  to  the other function. So it is useful when one studies the smoothing property of nonlinear Schr\"odingers of power type. The range on $s$ is sharp, since \eqref{bilinear} fails for $|s| >1-2/r$ (see the discussion below Proposition \ref{mixed-bi}).
The  estimate \eqref{bilinear} is strongly connected to the bilinear restriction estimates for the paraboloid (see \cite{kv, lv, t5}). In fact, for the proof of Theorem \ref{hilbert}  we establish estimates for bilinear interactions between waves at different frequency. It relies on the argument used to prove the bilinear restriction estimate for the paraboloid \cite{lv, t5}, which makes use of wave packet decomposition and induction on scaling (see Proposition \ref{mixed-bi} and Corollary \ref{interaction} below).

\smallskip

Aside from the power type, one of the most typical nonlinearity is that of Hartree in the study of nonlinear Schr\"odiner equations (see \eqref{hartree} below). To handle the Hartree type nonlinearity, we consider the  trilinear operator $\mathcal H$ which is given by
\[\mathcal H (f,g,h) \equiv \tfg.\]
Here $|\nabla |^{2-n}$ is the pseudo-differential operator with symbol $|\xi|^{2-n}$ which is the convolution with $c_n|x|^{-2}$. To make the operator have sense, we assume $n\ge 3$ throughout the paper when we use  the notation $|\nabla |^{2-n}$.  As it is turned out (see Theorem \ref{hart} and Theorem \ref{T}),
 the trilinear estimate enables us to control  the interaction of waves arising in Hartree type nonlinearity more effectively. It is stated as follows:
\begin{thm}\label{trid} Let $n\ge 3$ and let  $(\tilde q, \tilde r)$ be admissible. Suppose that $\sss$ are positive numbers satisfying $\ssum$. Then, if $s_3>\frac12$
\begin{align}\label{trid-1}\|\mathcal H (f,g,\nabla h)\|_{\mix{\tilde q'}{\tilde r'}}\lesssim \hnormfg.\end{align}
If $s_1>\frac12$, then
\begin{align}\label{trid-2}\|\mathcal H (\nabla f,g,h)\|_{\mix{\tilde q'}{\tilde r'}}\lesssim \hnormfg.\end{align}
\end{thm}

It should be noticed that the estimates \eqref{trid-1}, \eqref{trid-2} are invariant under scaling for all admissible  $(\tilde q, \tilde r)$ (cf. Lemma \ref{stricha}).
For the proof we first show frequency localized estimates (Proposition \ref{har-key-prop} below)  which also rely on the bilinear interaction estimates and the scaling structure of $\mathcal H$.  Compared with \eqref{bilinear}, a stronger interaction estimate is possible thanks to  operator the $|\nabla|^{2-n}$ which  gives additional decay in frequency space.

\smallskip

Now we  consider
applications of Theorem \ref{hilbert} and \ref{trid} to nonlinear Schr\"{o}dinger equations. We are concerned with the Cauchy problem of $L^2$  critical nonlinear Schr\"{o}dinger equation in $\mathbb{R}^n, n \ge 3$, of which nonlinear part is given by the nonlinear potential $V(u)$ of Hartree  or power type:
\begin{eqnarray}\label{hartree}
\begin{cases}
iu_{t} + \Delta u =   V(u)u,\;\; (t,x) \in \mathbb [0,T] \times \mathbb{R}^{n},\; T>0,\\
u(0,x) = u_0(x) \in H^s(\mathbb{R}^n).
\end{cases}
\end{eqnarray}
That is to say, $V (u) = \kappa|x|^{-2}\ast |u|^2$ or $V(u) = \kappa |u|^\frac{4}{n}$ with $\kappa = \pm 1$.
Here $u:[0,T] \times\mathbb{R}^{n} \to \mathbb{C}$ is a complex valued function.
If $u$ is a solution to \eqref{hartree}, the scaled function $\lambda^{\frac n2}u(\lambda^2t, \lambda x),$ $\la > 0$ is also a solution. Hence \eqref{hartree} is invariant under the scaling in $L^2$ space (i.e. $L^2$ critical).  By the Duhamel's principle the problem \eqref{hartree} is equivalent to solving the integral equation for $t \in [0,T]$;
\begin{align}\label{inthartree}
u(t) = e^{it\Delta}u_0 - i\int_0^t e^{it(t-t')\Delta} (V(u)u)(t')\,dt'.
\end{align}
It is well known that the problem \eqref{hartree} is locally wellposed for every $s\ge 0$ (see \cite{ca, tb}). The lifespan of solution $u$ depends on the $H^s$ norm if $u_0 \in H^s, s > 0$, and the profile of $u_0$ if $u_0 \in L^2$, respectively.
The solution $u \in C([0,T]; H^s)$ to \eqref{hartree} satisfies conservation laws, namely, mass and energy;  for any $t \in [0,T]$, if $s \ge 0$
\[
\|u(t)\|_{L^2}^2 = \|u_0\|_{L^2}^2,\;\;\]
and if $s\ge 1$
\[E(u(t)) \equiv \frac12\|\na u(t)\|_{L^2}^2 + \omega \int V(u(t))|u(t)|^2\,dx= E(u_0)\textcolor{red}{,}
\]
where $\omega = 1/4$ if $V (u) = \kappa|x|^{-2}\ast |u|^2$ and $\omega = \frac{n}{4+2n}$ if $V(u) = \kappa |u|^\frac{4}{n}$. If the data is sufficiently smooth ($s \ge 1$), various results were established by using the classical energy argument. However, it does not work any longer when $0 \le s < 1$  and there has been a lot of works devoted to extending those results to lower regularity initial data (for instance see \cite{bo, ckstt04, cr}).

\smallskip

We firstly apply Theorem \ref{hilbert}, \ref{trid} to study the smoothing properties  of solutions to the Cauchy problem \eqref{hartree}.
We consider a strong global (in $x$-space) smoothing effect such that the Duhamel's part \begin{equation}\label{duhamel}
D(t) \equiv u(t) - \schr u_0 \in C([0,T]; H^1)
\end{equation}  for all $T$ within the lifespan when the initial data $u_0$ is in $H^s$, $0\le s<1$.
The smoothing actually stems from the interaction of Schr\"{o}dinger waves arising in the nonlinear term. It  was first observed by Bourgain \cite{bo} for $V(u) = \kappa|u|^2, n = 2, s > 2/3$ and later extended by Keranni and Vargas \cite{kv} for $V(u) = \kappa|u|^\frac4n, n \ge 1, s > s_n$, where $s_1 = 3/4, s_n = n/(n+2)$ for $2 \le n \le 4$, $s_n = (n^2+2n-8)/n(n+2)$. To utilize the interaction, they established refined bilinear Strichartz estimates in $L_{t,\,x}^\frac{n+2}{n}$ as mentioned above.
In the following, we get better smoothing effects that \eqref{duhamel} holds for a rougher $u_0$, using the Theorems \ref{hilbert} and \ref{trid} together with the duality arguments based on the Bourgain space (\cite{bo,kv}).


\begin{thm}\label{hart} Let $n\ge 3$.
$(1)$ If $u_0\in H^{s}(\bbr ^n)$ and $1/2 < s < 1$, then there is a maximal existence time $T^*>0$ such that a unique solution $u$ to \eqref{hartree} with $V(u) = \kappa |x|^{-2}*|u|^2$ exists in $C([0,T^*); H^s)$ and $D$ satisfies \eqref{duhamel} for all $T<T^*$.

\noindent$(2)$ Let $s_n = \frac12$ for
$n = 3, 4$ and $s_n=1-\frac{8}{n^2}$ for
$n \ge 5$. If\; $u_0\in H^{s}(\bbr ^s)$, $s_n < s < 1$, then there is a maximal existence time $T^*$ such that a unique solution $u$ in $C([0,T^*); H^s)$ to  \eqref{hartree} with $V(u) = \kappa |u|^\frac4n$ and $D$ satisfies \eqref{duhamel} for all $T<T^*$.
\end{thm}
\noindent In part (2) we do not have any improvement on 2-d result
which was obtained in \cite{kv} ($s_2=\frac12$). The above result
shows that  the Hartree type interaction is more effective than the
power type when $n \ge 5$, which may  be interpreted as weaker (of lower power)
nonlinearity causes a lower interaction between the waves.
The smoothing effect can be used to show an $H^1$ mechanism for the blowup phenomenon of the Cauchy problem \eqref{hartree} (see Remark 1.3 of \cite{kv}). In \cite{kv},
it was shown that if $T^*$ is finite, then
\[\|\nabla D(t)\|_{L^2} \gtrsim (T^*-t)^{-1/2}\]
for power type NLS provided that \eqref{duhamel} holds for all
$T<T^*$. Hence, part (2) of Theorem \ref{hart} extends the possible
range of $s$. Similarly, using part (1) of Theorem \ref{hart}  and
the  argument in \cite{kv} together with well-known scaling
argument, one can also get the same blowup rate of $D(t)$ for the
finite time blowup solution of  Hartree type NLS as long as $u_0\in
H^{s}(\bbr ^n)$ and $1/2 < s < 1$.

\smallskip

We now consider  the global well-posedness of defocusing $L^2$
critical Hartree equation, \eqref{hartree} with $\kappa = +1$, for
rough initial data in $H^s, 0  < s < 1$. Recently Chae and Kwon
\cite{ck} considered the same problem \eqref{hartree} and they got global
well-posedness for $u_0\in H^s$, ${2(n-2)}/({3n-4}) < s  < 1 $.
Their result is based on the  so called $I$-method. (For details and
recent development of $I$-method, we refer readers to \cite{cgt,
ckstt02, ckstt04, ckstt08, cr, crsw, dpst, fg}.) We here make
further improvement. By exploiting the interaction
of Schr\"odinger waves systematically (Proposition
\ref{har-key-prop}), we obtain better decay estimates for almost
energy conservation and interaction Morawetz inequality (see
Proposition \ref{ACL}, \ref{ebound})  which are the major estimates
for $I$-method. As a consequence we get the following global
well-posedness theorem.
\begin{thm}\label{T} Let $n\ge 3$ and $V(u) = |x|^{-2}*|u|^2$.
 Then the initial value problem of \eqref{hartree} is
globally well-posed for data $u_0\in H^s(\bbr^n)$ when $ \frac{4(n-2)}{7n-8} < s <1$.
\end{thm}
\noindent The global well-posedness  for the spherically symmetric data in $L^2$  was shown by  Miao, Xu and Zhao \cite{mxz}. They  adopted the method due to  Killip, Tao and Visan \cite{ktv}. For the 2-d cubic NLS,   Colliander and Roy \cite{cr} recently combined the improved estimate in \cite{ckstt08} with a Mowawetz error estimate by using the double layer bootstrap in time,
 and established the global well-posedness for the $L^2$ critical NLS on $\bbr^2$ with data in  $H^s$, $s> 1/3$. It seems highly possible that such approach also makes further progress for the Hartree equations  if it is combined with the results of this paper. We hope to address such issues somewhere else. Compared to the previous works, our proof of almost energy conservation and  interaction Morawetz inequality is more systematic and flexible. We believe that it may be useful in studies of  various related problems.

\smallskip

This paper is organized as follows: In Section 2 we will obtain the
bilinear interaction estimate,  trilinear
Hartree type interaction estimate, and prove Theorem \ref{hilbert}, \ref{trid}. In Section 3
we will show the local well-posedness and smoothing effect of
Duhamel's part of solutions to \eqref{hartree}. The Section 4 is
devoted to showing the global well-posedness of defocusing Hartree
equation. Lastly  we append  a brief introduction to wave packet
decomposition of Schr\"odinger wave, which will be used in Section
2.

\smallskip

We now list the notations
which are frequently used  in the paper:

\smallskip

\noindent $\bullet$ $A \lesssim B$ means that  $A \le C B$ for some
constant $C > 0$ which may vary from lines to lines. We
also write $A \sim B$ when $A \lesssim B$ and $B \lesssim A$.

\smallskip

\noindent$\bullet$ The symbol $\nabla$ denotes the gradient
$(\partial/\partial_1,\cdots,\partial/\partial_n)$ and $\Del$ the Laplacian $\nabla
\cdot \nabla = \sum_j \partial^2/\partial_j^2$. We also denote  $(-\Del)^\frac12$ by
$|\na|$.

\smallskip

\noindent$\bullet$ Let $J_T$ be the time interval $[0,T]$.
For a measurable function $F$ the mixed norm 
is defined by $\|F\|_{\mix{q}{r}(J_T\times \bbr^n)} =
\bigg(\int_{J_T}\bigg(\int_{\bbr^n}|F(t,x)|^r\,dx\bigg)^\frac{q}{r}\bigg)^\frac1q.$
We use $\|F\|_{\mix{q}{r}}$ to denote $\|F\|_{L_{t,x}^q(\bbr \times \bbr^n)}$. $L^p$ is the usual Lebesgue space $L_x^p(\bbr^n)$.

\smallskip

\noindent$\bullet$ The Fourier transform of $f$ is defined by
$\mathcal{F}(f)(\xi) = \widehat f(\xi) \equiv \int e^{-ix\cdot
\xi}f(x)\,dx$ and its inverse by $\mathcal{F}^{-1}(g)(x) \equiv
(2\pi)^{-n}\int e^{i x\cdot \xi}g(\xi)\,d\xi$. Hence $\schr f(x) =
\mathcal{F}^{-1}(e^{-it|\cdot|^2}\mathcal{F}(f))(x)$ $=
(2\pi)^{-n}\int_{\mathbb R^n} e^{i(x\cdot\xi-t|\xi|^2)} \widehat
f(\xi) d\xi$.

\smallskip
\noindent$\bullet$ Let $N$ denote dyadic number and let $P_N$ be the Littlewood-Paley projection
operator with symbol $\chi(\xi/N) \in C_0^\infty$ supported in the
annulus $A(N) = \{1/2N \le |\xi| \le 2N\}$ such that $\sum_N P_N=id$. We also define
$\widetilde P_1 = id - \sum_{N > 1}P_N$.

\smallskip

\noindent$\bullet$ The inhomogeneous Sobolev space $H^s (=
H^s(\mathbb{R}^n), s \in \mathbb{R})$ denotes  the space $\{f \in \mathcal S' :
\|f\|_{H^s} < \infty \}$, where $\|f\|_{H^s} \equiv (\sum_{N \ge
0}N^{2s}\|P_N f\|_{L^2}^2+\|f\|_2^2)^\frac12  \sim \|\langle \nabla \rangle^s
f\|_{L^2} = \left(\int \langle \xi \rangle^{2s}|\widehat
f(\xi)|^2\,d\xi \right)^\frac12$. Here $\langle A \rangle=\sqrt{1 + |A|^2}$. We will also use the homogeneous Sobolev space
$\dot H^s = \{f \in \mathcal S'/\mathcal P : \|f\|_{\dot Hs} <
\infty\}$, where $\mathcal P$ is the totality of polynomials and the
seminorm $\|f\|_{\dot H^s} \equiv (\sum_{N: dyadic}N^{2s}\|P_N
f\|_{L^2}^2)^\frac12 \sim \||\nabla|^s f\|_{L^2}$. We note here that
if $|s| < n/2$, the definition of $\dot H^s$ makes sense in
$\mathcal S'$ and $C_0^\infty$ is dense in $\dot H^s$ (cf. \cite{p}).

\section{Mixed norm interaction estimates for the Schr\"odinger waves}
In this section we prove first bilinear interaction estimates  for
the Schr\"odinger waves. Considering the mixed norm space, it is
possible to get a better interaction estimate than the one obtained
in \cite{kv}. We denote by $B(\xi,\rho)$ the ball centered at $\xi$
with radius $\rho$.

\begin{prop}\label{mixed-bi} Let $n\ge 2$.
Suppose that $\supp\widehat f\subset B(\xi_0,\rho_1)$ and
$\supp\widehat g\subset B(\eta_0,\rho_2)$ for some $|\xi_0|,$
$|\eta_0|\le 1$. If $|\xi_0-\eta_0|\sim 1$ and $0<\rho_1,\rho_2\ll
1$, then for $\epsilon>0$ and $(q,r)$ satisfying that $r\le 4$, $2<q$
and  $1-\frac2r\le \frac2q<(n+1)(\frac12-\frac1r)$
\[\|\schr f\schr g\mixednorm \lesssim \min(\rho_1,\rho_2)^{\alpha(q,r)-\epsilon} \|f\|_{L^2}\|g\|_{L^2} \]
where $\alpha(q,r)=(n+1)(1-2/r)-4/q.$
\end{prop}

It can be shown that the bounds in the above estimates are sharp up
to $\epsilon$. Indeed, assuming $\rho_1 \le \rho_2$,  let us
consider the functions $f$ and $g$ given by  $\widehat f = \chi_A$ and
$\widehat g = \chi_B$ for $A=\{\xi: |\xi_1-1|\le \rho_1^2, \
|\xi_i|\le \rho_1, i=2,\dots, n\}$ and $B=\{\xi: |\xi_1+1|\le
\rho_1^2, \ |\xi_i|\le \rho_1, i=2,\dots, n\}$.  Then it is easy to
see that $|\schr f(x)|,$ $|\schr g(x)|\ge c\rho_1^{n+1}$ if $|x_1|,
|t|\le c\rho_1^{-2}$ and $|x_i|\le c\rho_1^{-1}$ for some $c>0$,
$i=2,\dots, n$. Hence
\[\rho_1^{n+1-\frac{2(n+1)}r-\frac4q}\lesssim
\frac{\|\schr f\schr g\mixednorm}{\|f\|_{L^2}\|g\|_{L^2}}.
\] Letting
$\rho_1\to 0$ we can see that this implies $\alpha(q,r)-\epsilon\le
n+1-\frac{2(n+1)}r-\frac4q$ for any $\epsilon>0$. It also shows the
failure of the estimates when $\frac2q>(n+1)(\frac12-\frac1r)$. The
example above is actually the squashed cap function which was used
to show the sharpness of bilinear restriction estimates \cite{tvv}.

\begin{rem}  Using \eqref{l22} below and Plancherel's theorem, we can show that for $(q,r)$ satisfying $1-2/r\ge 2/q$ and $r\le 4$,
\[\|\schr f\schr g\mixednorm \lesssim \min(\rho_1,\rho_2)^{n(1-\frac2r)-\frac2q} \|f\|_{L^2}\|g\|_{L^2}.\]
It is also sharp as it can be shown by using the functions $f$ and $g$ with  $\widehat f = \chi_A$ and $\widehat g = \chi_B$ for
$A=\{\xi: |\xi-e_1|\le \rho_1\}$ and $B=\{\xi: |\xi_1+e_1|\le \rho_1\}$
\end{rem}

For the proof of Proposition \ref{mixed-bi} we will use the wave
packet decomposition for the Schr\"odinger operator. Such
decomposition was used to study Fourier restriction estimates
\cite{lv, t5, wo}.

\subsection*{Proof of Proposition \ref{mixed-bi}} By symmetry we may assume $\rho_1\le \rho_2$.
We start with recalling the estimates
\begin{align}\label{bimixed}
\|\schr f\schr g\|_{\mix {q/2} {r/2}}\lesssim \|f\|_{L^2}\|g\|_{L^2}
\end{align}
for $\frac2q<(n+1)(\frac12-\frac1r)$, $2 < q,r \le 4$. See Theorem 2.3 of \cite{lv}.
Also we make use of the estimate
\begin{align}
\label{l22}
\|\schr f\schr g\|_{\mix {2} {2}}&\lesssim \rho_1^{\frac{n-1}2} \|f\|_{L^2}\|g\|_{L^2}\footnote{It  actually reads as
$\|\schr f\schr g\|_{\mix {2} {2}}\lesssim \min(\rho_1, \rho_2)^{\frac{n-1}2} \|f\|_{L^2}\|g\|_{L^2}$ since we are assuming $\rho_1\le \rho_2$.},
\end{align}
which already appeared in several literatures (for instance see
\cite{bo} and \cite{kv}). For the convenience of reader we give a
simple proof based on Plancherel's theorem.

Using an affine transformation we may assume $\xi_0=0$. By decomposing  the  Fourier support of $g$ into finite number of sets, rotation and  dilation, it is enough to show \eqref{l22}
whenever $f$ and $g$ are Fourier-supported in $B(0,\rho_1)$ and $B(e_1, \delta)$ for some $0 < \delta \ll1$, respectively.
We write
\[ \schr f(x)\schr g(x)=\int e^{i(x \cdot (\xi+\eta)-t(|\xi|^2+|\eta|^2))} \widehat f (\xi)\widehat g(\eta) d\xi d\eta.\]
Freezing $\bar\xi=(\xi_2,\dots,\xi_n)$, we  consider a bilinear operator
\[B_{\bar\xi}(f,g)=\int e^{i(x\cdot (\xi+\eta)-t(|\xi|^2+|\eta|^2))} \widehat f (\xi_1,\bar\xi)\widehat g(\eta) d\xi_1 d\eta.\]
We make the change of variables
$\zeta=(\zeta_1, \zeta_2,\dots, \zeta_{n+1})=(\xi+\eta,|\xi|^2+|\eta|^2).$
Then by direct computation  one can see that
$\left|\frac{\partial \zeta}
{\partial(\xi,\eta_1)}\right|=2|\xi_1-\eta_1|\sim 1$
on the supports of $\widehat f$ and $\widehat g$. Hence making the change of variables $(\xi_1,\eta)\to \zeta$, applying Plancherel's theorem and
reversing the change variables  ($\zeta \to (\xi_1,\eta)$), we have
\[\|B_{\bar\eta}(f,g)\|_{\mix 2 2}\lesssim\|\widehat f(\xi_1, \bar \xi) \widehat g(\eta)\|_{\mi 2}. \]
Since $\schr f(x)\schr g(x)=\int  B_{\bar\xi}(\hat f(\cdot, \bar\xi),\hat g(\cdot)) d\bar\xi,$
by Minkowski's inequality we get
\[\|\schr f\schr g\|_{\mix 2 2} \lesssim \int  \|B_{\bar\xi}(\hat f(\cdot,\bar \xi)\hat g(\cdot))\|_{\mi 2}
d\bar\xi. \]
This gives the desired estimate \eqref{l22} by Cauchy-Schwarz inequality because $|\bar \xi|\le \rho_1$.

\smallskip

When $n=2, 3$ we only need to interpolate
\eqref{bimixed}, \eqref{l22} and $\|\schr f\schr g\|_{\mix {\infty} {1}}\lesssim \|f\|_2\|g\|_2$ for the proof of the theorem. It gives all the desired estimates.
Hence, similarly when $n \ge 4$,  it is sufficient to show that for $\epsilon>0$
\begin{align}
\label{l12}
\|\schr f\schr g\|_{\mix {q_\epsilon/2} {r_\epsilon/2}}&\lesssim \rho_1^{\frac{n-3}2-\epsilon}\|f\|_{L^2}\|g\|_{L^2}.
\end{align}
Here $(q_\epsilon, r_\epsilon)$ converges to $(2,4)$ as $\epsilon\to 0$. A similar estimate already appeared in
\cite{lv} for the wave operator and its proof is based on the induction on scale argument. We also follow the same lines of argument.

Let $\lambda$ be a large number so that $\lambda \gg \rho_1^{-\frac12}$
and let us set $\mathcal Q(\lambda)=Q(\lambda)\times (-\lambda,\lambda)$, where  $Q(\lambda)$ is the cube centered at the origin with side length $2\lambda$.
We make an assumption that
\begin{equation}\label{induct} \|\schr f\schr g\|_{\mix {1}{2}
(\mathcal Q(\lambda))}\lesssim \rho_1^{\frac{n-3}2}\lambda^\alpha \|f\|_{L^2}\|g\|_{L^2}.
\end{equation}
Due to $\eqref{l22}$ and H\"older's inequality  the above is  valid with $\alpha=1/2$.
Now we attempt to suppress $\alpha$ as small as possible.

Let $\{b\}$ be the collection
of the $\lambda^{1-\delta}$-cubes $b$ partitioning
$\mathcal Q(\lambda)$. We make use of the wave packet decomposition and Lemma \ref{wavepacket} which had crucial role in the proof of the sharp bilinear restriction estimates for the paraboloids \cite{t5}. We provided some basic properties of wave packets in the appendix.
Using wave packet decomposition at scale $\lambda$ and the triangle inequality, we have
\[
\|\schr f\schr g\|_{\mix {1}{2}
(\mathcal Q(\lambda))}\le \sum_{b} \|\sum_{T,T'} \schr f_T\schr g_{T'}\|_{\mix {1}{2} ( b)}.
\]
Using the relation $\sim$, we break the mixed integration over $b$ so that
\[
 \|\schr f\schr g\|_{\mix {1}{2} (\mathcal Q(\lambda))} \le  I+ I\!I,\]
where
 \begin{align*}
  I&=\sum_{b}\|\sum_{T\sim b \text{ and } T' \sim b}
e^{it\Delta}f_T e^{it\Delta}g_{T'}\|_{\mix {1}{2}  ( b)}, \\
I\!I&=\sum_{b}\|\sum_{T\not \sim b \text{ or } T'\not \sim b}
e^{it\Delta}f_T e^{it\Delta}g_{T'}\|_{\mix {1}{2}  ( b)}.
 \end{align*}
For the first we use the induction assumption \eqref{induct}  to get
\[I\le C\rho_1^{\frac{n-3}2}\lambda^{\alpha(1-\delta)}\sum_{b}\| \sum_{T\sim b} f_T\|_{L^2}\| \sum_{T\sim b} g_T\|_{L^2}\]
 because $b$ is a cube of size $\sim \lambda^{1-\delta}$.  Hence by \eqref{l2sum} and Schwarz's inequality
\[
 I\le C \rho_1^{\frac{n-3}2}\lambda^\epsilon\lambda^{\alpha(1-\delta)}\|f\|_{L^2}\|g\|_{L^2}.
 \]

H\"older's inequality and  \eqref{notb} give
\[
\|\sum_{T\not \sim b \text{ or } T'\not \sim b}
e^{it\Delta}f_T e^{it\Delta}g_{T'}\|_{\mix 1 2(b)}\le C\lambda^{c\delta-(n-3)/4} \|f\|_{L^2}\|g\|_{L^2}.
\]
Since there are only $\lambda^{c\delta}$-cubes $b$ and $\rho_1^2\gg \lambda^{-1}$,  it follows that
\[ I\!I\lesssim\rho_1^{\frac{n-3}2}\lambda^\epsilon\lambda^{c\delta} \|f\|_{L^2}\|g\|_{L^2}.\]
Combining two estimates for $I$ and $I\!I$, we get
\begin{equation}\label{ind}\|\schr f\schr g\|_{\mix {1}{2}
(\mathcal Q(\lambda))}\lesssim\rho_1^{\frac{n-3}2}(\lambda^{(1-\delta)\alpha} + \lambda^\epsilon\lambda^{c\delta}) \|f\|_{L^2}\|g\|_{L^2}.
\end{equation}

Therefore we see that the assumption \eqref{induct} implies the above estimate \eqref{ind}. Since $\epsilon, \delta>0 $  can be chosen to be arbitrarily small,
we get for any $\alpha>0$
\[
\|\schr f\schr g\|_{\mix {1}{2}
(\mathcal Q(\lambda))}\lesssim\rho_1^{\frac{n-3}2}\lambda^\alpha \|f\|_{L^2}\|g\|_{L^2}
\]
 by iterating this implication \eqref{induct} $\to$ \eqref{ind}
finitely many times\footnote[2]{For this one should note that the constant $c$ in \eqref{ind} is independent of $\epsilon, \lambda$.}.
To ungrade  this  to the global one, we need the following globalization lemma in \cite{lv}.

\begin{lem}\label{global} Let $S_1$ and $S_2$ be compact surfaces with
boundary $S_i=\{(\xi,\phi_i(\xi))\;: \;\;\xi\in U_i\}$ and the
induced Lebesgue measures $d\sigma_i(\xi)=d\xi,$ $i=1,2$, which
satisfy $\|d\sigma_i\|\lesssim_i$, $\sigma_i(B(z, \rho))\lesssim\rho^{n-1}$ for any $z,\rho>0$ and
$|\widehat{d\sigma_i}(x,t)|\le C_i(1+|x|+|t|)^{-\sigma}$ for some
$C_i\ge 1$ and  $\sigma > 0$. Suppose that for some
$\frac{2+2\sigma}\sigma\ge q_0,r_0\ge 1$ and  $0<\epsilon\ll
\sigma$,
\[
\|\prod_{i=1}^2\widehat{f_id\sigma_i}\|_{L_t^{q_0}L_x^{r_0}(\mathcal Q(\lambda))}
\le C_0\lambda^\epsilon \prod_{i=1}^2 \|f_i\|_{L^2(d\sigma_i)}.
\]
   Let
$
\frac1{q_1}=\frac1{q_0}-
\frac{2\epsilon}{2\epsilon+\sigma}\big(\frac1{q_0}-\frac{\sigma}{2(\sigma+1)}\big),$
$
\frac1{r_1}=\frac1{r_0}-
\frac{2\epsilon}{2\epsilon+\sigma}\big(\frac1{r_0}-\frac{\sigma}{2(\sigma+1)}\big).
$
Then, for $q>q_1$
\[
\|\prod_{i=1}^2\widehat{f_id\sigma_i}\|_{L_t^{q}L_x^{r_1}} \lesssim
C_0^{1-\frac\epsilon\sigma}(\max(C_1,C_2))^{a\epsilon+(1-\frac{q_1}q)(1-\frac1{r_1})}\prod_{i=1}^2
\|f_i\|_{L^2(d\sigma_i)}
\]
with some $a>0$ depending on $\sigma.$
\end{lem}

Let us define two extension operators by
$$
\widehat {h d\sigma_1}=\int e^{i(x\cdot\xi-t|\xi|^2)} \beta(\frac{\xi-\xi_0}{\rho_1})h(\xi) d\xi, \;\;\;\; \widehat {h d\sigma_2}=\int e^{i(x\cdot\xi-t|\xi|^2)} \beta(\frac{\xi-\eta_0}{\rho_2})h(\xi) d\xi$$
for smooth $\beta$ supported in $B(0,2)$ and $\beta=1$ on $B(0,1)$.
Since  $\supp \widehat f \subset B(\xi_0,\rho_1)$ and  $\supp \widehat g \subset B(\eta_0,\rho_2)$,
by Plancherel's  theorem  it is sufficient to show that
the estimate
\[\|\prod_{i=1}^2 \widehat{h_id\sigma_i}\|_{\mix{1}{2}(\mathcal Q(\lambda))}\lesssim\rho_1^\frac{n-3}2\lambda^\alpha \prod_{i=1}^2\|h_i\|_2\]
implies the global estimate
\[\|\prod_{i=1}^2 \widehat{h_id\sigma_i}\|_{\mix {q(\alpha)/2}{r(\alpha)/2}}\lesssim\rho_1^{\frac{n-3}2-\epsilon(\alpha)}\prod_{i=1}^2\|h_i\|_2\]
with $q(\alpha)\to 1$,  $r(\alpha)\to 2$ and $\epsilon(\alpha)\to 0$ as $\alpha\to 0$. Hence using
Lemma \ref{global}, we only need to check that
\[ |\widehat{d\sigma_i}(x,t)|\lesssim(1+|x|+|t|)^{-\frac{n}2}.\]
This is easy to see by using stationary phase method because $\rho_1,\rho_2\ll 1$. It completes the proof of
Proposition \ref{mixed-bi}.

\subsection{Proof of Theorem \ref{hilbert}: Bilinear interaction estimates}

We note that the bilinear estimate in Proposition \ref{mixed-bi} is invariant under rescaling when  $2/q=n(1/2 - 1/r)$. Hence,  by Proposition \ref{mixed-bi} and rescaling it is easy to see the following Corollary \ref{interaction},  which shows that there is an interactive compensation  when one considers the Schr\"{o}dinger waves of different frequency levels.
Throughout the paper we denote by $A(\rho)$ the set $\{\xi:|\xi|\sim \rho \} $.

\begin{cor}\label{interaction} Let $n\ge 2.$
Let $(q,r)$ satisfy that $2/q=n(1/2 - 1/r)$,  $2 \le  r \le 4$. If $\supp \widehat f\subset A(N_1)$ and
$\supp \widehat g \subset A(N_2)$ for $0<N_1\le N_2$, then for any $\epsilon>0$,
\[ \|\schr f\schr g\|_{\mix {q/2} {r/2}}\lesssim \left(\frac{N_1}{N_2}\right)^{1-2/r-\epsilon}\|f\|_{L^2}\|g\|_{L^2}.\]
\end{cor}

We  now give  the proof of the Theorem \ref{hilbert}.
The assertion for $s = 0$ follows from the H\"{o}lder's inequality and Strichartz estimate. By symmetry we may assume that $s > 0$.  Let $P_N$ be the Littlewood-Paley projection as stated in the introduction. For simplicity we set $f_N = P_N f$ and
 break $\schr f \schr g$ so that
\[\schr f \schr g=\sum_{N_1,\,N_2:\, dyadic} \schr f_{N_1}(\schr g_{N_2}).\]
for any $f \in \dot H^s$ and $g \in \dot H^{-s}$.
Since $|\nabla|\sim N_2$ on the Fourier support of $g_{N_2}$, it is enough to show that
\[\|\sum_{N_1,\,N_2}\schr f_{N_{1}}\schr g_{N_2}\|_{\mix{q/2}{r/2}} \lesssim \bigg(\sum_{N_1} N_1^{2s}\|f_{N_1}\|_{L^2}^2\bigg)^\frac12 \bigg(\sum_{N_2} N_2^{-2s}\|g_{N_2}\|_{L^2}^2\bigg)^\frac12.\]

Let us set $N_{12}=N_1N_2$. By the triangle inequality
\[\|\sum_{N_1,\,N_2} \schr f_{N_{1}}\schr g_{N_2}\|_{\mix{q/2}{r/2}}\le I+I\!I,\]
where
$$
I=\|\sum_{N_1\ge 1}\sum_{N_2} \schr f_{N_{12}}\schr g_{N_2}{\mixednorm},\;\;I\!I=\|\sum_{N_1<1}\sum_{N_2}\schr f_{N_{12}}\schr g_{N_2}{\mixednorm}.$$
Since $I\lesssim \sum_{N_1\ge 1}N_1^{2/r-1+\epsilon} \sum_{N_2}  \|f_{N_{12}}\|_{L^2}\|g_{N_2}\|_{L^2} $ by the triangle inequality and Corollary \ref{interaction}, we see that
\begin{align*}
I
&\lesssim \sum_{N_1\ge 1} N_1^{2/r-1-s+\epsilon}
            \sum_{N_2}  (N_{12})^{s}\|f_{N_{12}}\|_{L^2} N_2^{-s}\|g_{N_2}\|_{L^2}\\
&\lesssim \sum_{N_1\ge 1} N_1^{2/r-1-s+\epsilon}
            \big(\sum_{N_1} N_1^{2s}\|f_{N_1}\|_{L^2}^2\big)^\frac12
                \big(\sum_{N_2} N_2^{-2s}\|g_{N_2}\|_{L^2}^2\big)^\frac12 \\
&\lesssim \big(\sum_{N_1} N_1^{2s }\|f_{N_1}\|_{L^2}^2\big)^\frac12
                \big(\sum_{N_2} N_2^{-2s }\|g_{N_2}\|_{L^2}^2\big)^\frac12,
\end{align*}
provided $2/r-1-s+\epsilon<0$. We now turn to $I\!I$. By the triangle inequality and Corollary \ref{interaction}, $I\!I\lesssim \sum_{N_1<1}N_1^{1-2/r-\epsilon} \sum_{N_2}  \|f_{N_{12}}\|_{L^2}\|g_{N_2}\|_{L^2}.$ Hence, by Schwarz's inequality
\begin{align*}
I\!I
&\lesssim \sum_{N_1< 1} N_1^{1-2/r-\epsilon-s}
            \big(\sum_{N_1} N_1^{2s }\|f_{N_1}\|_{L^2}^2\big)^\frac12
                \big(\sum_{N_2} N_2^{-2s}\|g_{N_2}\|_{L^2}^2\big)^\frac12 \\
&\lesssim \big(\sum_{N_1} N_1^{2s }\|f_{N _1}\|_{L^2}^2\big)^\frac12
                \big(\sum_{N_2} N_2^{2(1-s) }\|g_{N_2}\|_{L^2}^2\big)^\frac12
\end{align*}
as long as  $1-2/r-\epsilon-s>0$. This completes the proof of Theorem \ref{hilbert}.

\subsection{Proof of Theorem \ref{trid}: Trilinear interaction of Hartree type nonlinearity}

First we recall the following which is a consequence of Strichartz estimate and Hardy-Littlewood -Sobolev inequality.
\begin{lem}\label{stricha}
For any admissible $(\tilde q,\tilde r)$,
\[\|\calh \amixednorm\lesssim \|f\|_{L^2}\|g\|_{L^2}\|h\|_{L^2}\]
and  the estimates are  invariant  under the rescaling
$(f,g,h)\to (f_\lambda, g_\lambda, h_\lambda)=(\lambda^\frac n2 f(\lambda \cdot),$ $\lambda^\frac n2 g(\lambda \cdot), \lambda^\frac n2 h(\lambda \cdot))$ for any $\lambda > 0$.
\end{lem}

To show this, observe that for any admissible $(\tilde q,\tilde r)$ there is an admissible $(q,r)$ such that $\left({1}/{\tilde q'},1/{\tilde r'}\right)+\left(0,({n-2})/n\right)=3\left(1/q,1/r\right).$ Then, using H\"older's and Hardy-Littlewood -Sobolev inequalities one can get the desired estimate.

Via frequency localization on annulus  we first obtain the following trilinear interaction estimate.

\begin{prop}\label{har-key-prop}
  Let $n\ge 3$ and $N_1,N_2,N_3$ be positive numbers. Suppose that $supp\, \widehat f$, $supp\, \widehat g$,  $supp \,\widehat h$  are contained in $A(N_1),$ $A(N_2),$ $A(N_3)$, respectively.
   Then for any admissible pair $(\widetilde q, \widetilde r)$,
  \begin{eqnarray}\label{hart-key} \|\calh\|_{\amixed}\lesssim C(N_1, N_2,N_3) \normfg,\end{eqnarray}
  where
\[ C(N_1,N_2,N_3)=\bigg(\;\frac{\min(N_1, N_2, N_3)}{\max(N_1, N_2, N_3)}\;\bigg)^{1/2}.
\]
 \end{prop}

For the proof  it is enough to consider two endpoints $(\widetilde q', \widetilde r')=(1,2), \ (2,\frac{2n}{n+2})$ because interpolation gives the remaining estimates. By symmetry we may assume that $N_1\ge N_2$. On account of  scaling structure (Lemma \ref{stricha})
we  may also assume that
\[1=\max(N_1, N_3).\]
Hence we  can further assume that $1\gg \min(N_1,N_2,N_3)$ since the desired estimates are already contained in Lemma \ref{stricha}
when $N_1\sim N_2\sim N_3$. Then we prove  Proposition \ref{har-key-prop} by considering the cases $N_1\gg N_2$ and $N_1\sim N_2$, separately. To begin with, we recall the following simple lemma which can be easily shown by using
the Strichartz estimates and rescaling.
\begin{lem}
\label{str}
  If $\supp \widehat f\subset A(N)$, for $q,r\ge 2$ satisfying $n/r+2/q\le n/2$
\[
\|\schr f\|_{\mix q r}\lesssim N^{\frac n2-\frac nr-\frac2q}\|f\|_2.
\]
\end{lem}

\noindent{\textbf{Case $N_1\gg  N_2$.}} In this case  the spatial Fourier support of $\schr f \schr g$ is contained in $A(2N_1)$ because $N_1\gg  N_2$. Hence,  $\fd{2-n}\sim N_1^{2-n}$.  Using H\"ormander-Mikhlin multiplier theorem we see
\begin{equation}
\label{break}
\|\calh\|_{\amixed}\lesssim N_1^{2-n}\|\schr f\schr g\schr h\|_{\mix{\tilde q'}{\tilde r'}}.
\end{equation}
We now have three subcases $(i) \; 1=N_3\ge N_1\gg  N_2$,  $(ii) \; 1=N_1\ge N_3\ge  N_2$,  $(iii) \; 1= N_1\gg  N_2\ge N_3$.

We consider the case $(i)$ first.  Taking $(\widetilde q', \widetilde r')=(2,\frac{2n}{n+2})$ in \eqref{break} and using H\"older's inequality,  it follows that
$$
\|\calh\|_{\mix{2}{\frac{2n}{n+2}}}
\lesssim  N_1^{2-n}\|\schr g\schr h\|_{\mix 22}\|\schr f\|_{\mix{\infty}{n}}.
$$
Since $1\gg N_2$, using \eqref{l22} and Bernstein's inequality (or Lemma \ref{str}), we get
\[\|\calh\|_{\mix{2}{\frac{2n}{n+2}}}
\lesssim  N_2^{\frac12}\bigg(\frac{N_2}{N_1}\bigg)^\frac {n-2}2\normfg.\]
This gives the desired estimate for $(\widetilde q', \widetilde r')=(2,\frac{2n}{n+2})$.
Similarly, taking $(\widetilde q', \widetilde r')=(2,1)$ in \eqref{break} and using H\"older's inequality,  we have
$$
\|\calh\|_{\mix{1}{2}}
\lesssim  N_1^{2-n}\|\schr f\|_{\mix{2}{\infty}}\|\schr g\schr h\|_{\mix {2}2}.
$$
By \eqref{l22} and Lemma \ref{str}, we get
\begin{align*}
\|\calh\|_{\mix{1}{2}}\lesssim  N_2^{\frac12}\bigg(\frac{N_2}{N_1}\bigg)^\frac {n-2}2\normfg.
\end{align*}
Hence we get the desired bound for  $(\widetilde q', \widetilde r')=(2,1)$ because $N_2\le N_1$.

The remaining two cases $(ii), (iii)$ can be handled similarly.
In fact, for the case   $(ii)$, repeating the same argument, using \eqref{break}, \eqref{l22} and Lemma \ref{str} we see that
\begin{align*}
\|\calh\|_{\mix{2}{\frac{2n}{n+2}}}
     \lesssim \|\schr f\schr g\|_{\mix 22}\|\schr h\|_{\mix{\infty}{n}}
\lesssim  N_2^\frac{n-1}2N_3^\frac{n-2}2 \normfg
\end{align*}
and
\begin{align*}
\|\calh\|_{\mix{1}{2}}
      \lesssim \|\schr f\schr g\|_{\mix 22}\|\schr h\|_{\mix{2}{\infty}}
\lesssim  N_2^\frac{n-1}2N_3^\frac{n-2}2 \normfg.
\end{align*}
Because $1\gg N_2$. So  we get the desired estimate for $(\widetilde q', \widetilde r')=(1,2), (2,\frac{2n}{n+2})$.

Finally, for case  $(iii)$,  then by \eqref{break} and repeating the same argument
one can show that for  $(\widetilde q', \widetilde r')=(1,2), \ (2,\frac{2n}{n+2}),$
\begin{align*}
\|\calh\|_{\mix{\widetilde q'}{\widetilde r'}}\lesssim  N_3^\frac{n-1}2N_2^\frac{n-2}2 \normfg.
\end{align*}
This completes the proof for the case $N_1\gg  N_2$. Now we turn to the remaining case $N_1\sim N_2$.

\

\noindent{\textbf{Case $N_1\sim N_2$.}} In this case  $\fd{2-n}$ can not be handled simply as before. So we need an additional argument to handle this.  We begin with decomposing $\fd{2-n}$ so that
\[\fd{2-n}=\sum_{N:\text{dyadic}} N^{2-n} \psi( |\nabla|/{N})\]
with a cut-off $\psi$ supported in $A(1)$
\footnote{Actually the sum is taken over $N\lesssim \max(N_1,N_2)$ because of the supports of $\widehat f$, $\widehat g$.}.
Here $m(|\na|)$ is the multiplier operator defined  by $m(|\nabla|) f=\mathcal F^{-1}(m(|\cdot|)\widehat f)$ for a measurable  function $m$.
Then we have
\begin{equation}\label{decomp}
\calh=\sum_{N:\text{dyadic}} N^{2-n} \psi( |\nabla|/{N})(\schr f\schr g)\schr h.
\end{equation}
We first try to obtain estimates for $\psi( |\nabla|/{N})(\schr f\schr g)\schr h$.
We claim that for $(\widetilde q', \widetilde r')=(1,2), \ (2,\frac{2n}{n+2})$,
\begin{align}\begin{aligned}\label{claim}
\|\psi( |\nabla|/{N})&(\schr f\schr g)\schr h\|_{\mix{\widetilde q'}{\widetilde r'}}\\
&\lesssim N^{\frac{n-2}2}\big(\min(N,N_3)\big)^\frac{n-1}2\normfg.
\end{aligned}\end{align}

To show the claim  we break $f$ and $g$ into functions having Fourier supports in cubes of side length $2^{-2}N$.  Let $\{Q\}$ be a  collection of (essentially) disjoint cubes of side length $2^{-2}N$ covering
$A(N_1)$ and we  set
\[\widehat f_Q=\chi_Q(\xi) \widehat f, \quad \widehat g_Q=\chi_Q(\xi) \widehat g.\]
Then  we have
$f=\sum_Q f_Q$ and $ g=\sum_Q g_Q$,
and we may assume that $Q\subset A(N_1)$ because $N_1\sim N_2$.
Then it follows that
\begin{align}\begin{aligned}\label{string}
\mbox{LHS of \eqref{claim}}
&\lesssim
\sum_{Q,Q'}\|\psi( |\nabla|/{N})(\schr f_Q\schr g_{Q'})\schr h\|_{\mix{\widetilde q'}{\widetilde r'}}
\\
&\lesssim\sum_{\dist(Q,-Q')
\le 4N}\|\psi( |\nabla|/{N})(\schr f_Q\schr g_{Q'})\schr h\|_{\mix{\widetilde q'}{\widetilde r'}}
\end{aligned}\end{align}
because $\psi( |\nabla|/{N})(\schr f_Q\schr g_{Q'})=0$ if $\dist(Q,-Q')> 4N$.
Hence  it is enough to show that  for $(\widetilde q', \widetilde r')=(1,2), \ (2,\frac{2n}{n+2})$,
\begin{align}\begin{aligned}\label{12jj}
\|\psi( |\nabla|/{N})(\schr & f_Q\schr g_{Q'})\schr h\|_{\mix{\widetilde q'}{\widetilde r'}}\\
& \lesssim N^{\frac{n-2}2}\big(\min(N,N_3)\big)^\frac{n-1}2\|f_Q\|_{L^2}\|g_{Q'}\|_{L^2}\|h\|_{L^2}.
\end{aligned}\end{align}
Indeed, from \eqref{string} and \eqref{12jj} we get
\begin{align*}
\mbox{LHS of \eqref{claim}}
& \lesssim N^{\frac{n-2}2}\big(\min(N,N_3)\big)^\frac{n-1}2\sum_{\dist(Q,-Q')
\le 4N}\|f_Q\|_{L^2}\|g_{Q'}\|_{L^2}\|h\|_{L^2}
\\
& \lesssim N^{\frac{n-2}2}\big(\min(N,N_3)\big)^\frac{n-1}2
(\sum_{Q}\|f_Q\|_{L^2}^2)^\frac12(\sum_{Q'}\|g_{Q'}\|_{L^2}^2)^\frac12 \|h\|_{L^2}
\\
& \lesssim N^{\frac{n-2}2}\big(\min(N,N_3)\big)^\frac{n-1}2
\normfg.\nonumber
\end{align*}
For the second and third inequalities we used  Cauchy-Schwarz inequality and orthogonality, respectively.
Hence matters are reduced to showing \eqref{12jj}.

Now observe that
\begin{align*}
\psi( |\nabla|/{N})(\schr &f_Q\schr g_{Q'})=\iint e^{ix\cdot(\xi+\eta)-it(|\xi|^2+|\eta|^2)}\\
&
\times\phi(\xi/N-\xi_0)\psi((\xi+\eta)/N)\phi(\eta/N-\eta_0)
\widehat f_Q(\xi)\widehat g_{Q'}(\eta) d\xi d\eta
\end{align*}
for some $\xi_0,\eta_0\in \mathbb R^n$ and $\phi$ supported in $B(0,1)$.
Expanding $\Psi(\xi,\eta)=\phi(\xi-\xi_0)\psi(\xi+\eta)\phi(\eta-\eta_0)$ into
Fourier series on the cube of side length $2\pi$ which contains the support of $\Psi$, we have
\[\phi(\xi-\xi_0)\psi(\xi+\eta)\phi(\eta-\eta_0)=\sum_{k,\,l \,\in \,\mathbb{Z}^n} C_{k,\,l}\,e^{ i(k\cdot\xi+l\cdot\eta)}\]
with $\sum_{k,\,l}|C_{k,\,l}|\le C$, independent of $\xi_0,\eta_0$.
Plugging this in the above we get
\begin{align*}
&\qquad\psi( |\nabla|/{N})(\schr f_Q\schr g_{Q'})=\sum_{k,\,l\, \in \,\mathbb{Z}^n} C_{k,\,l}\,\schr f_Q^k \schr g_{Q'}^l
\end{align*}
with $\|f_Q^k\|_{L^2}=\|f_Q\|_{L^2}$ and $\|g_{Q'}^l\|_{L^2}=\|g_{Q'}\|_{L^2}$ for all $k,l$.
Hence to show \eqref{12jj}  it suffices to show that  for $(\widetilde q', \widetilde r')=(1,2), \ (2,\frac{2n}{n+2})$,
\[\|\schr f\schr g\schr h\|_{\mix{\widetilde q'}{\widetilde r'}}
\lesssim N^{\frac{2- n}2}\big(\min(N,N_3)\big)^\frac{n-1}2\|f\|_{L^2}\|g\|_{L^2}\|h\|_{L^2}\]
whenever $\widehat f,\widehat g$
are supported in cubes $Q, Q'\subset A(N_1)$ of side length $N$  and $\widehat h$ is supported in
$A(N_3)$. Note that $\dist(\supp \widehat f, \supp \widehat h)\sim 1$ and $\dist(\supp \widehat g, \supp \widehat h)\sim 1$ because there are only two possible cases $1=N_1\sim N_2\gg  N_3,$ $1=N_3\gg N_1\sim  N_2$. Hence, by H\"older's inequality, \eqref{l22}
and Lemma \ref{str}, we get
\begin{align*}
\|\schr f\schr g\schr h\|_{\mix{2}{\frac{2n}{n+2}}}
      &\lesssim \|\schr f\|_{\mix{\infty}{n}}\|\schr g\schr h\|_{\mix 22}
\\    &\lesssim N^{\frac{n-2}2}\big(\min(N,N_3)\big)^\frac{n-1}2 \normfg,
\end{align*}
and
\begin{align*}
\|\calh\|_{\mix{1}{2}}
      &\lesssim \|\schr f\schr g\|_{\mix 22}\|\schr h\|_{\mix{2}{\infty}}
\\    &\lesssim  N^{\frac{n-2}2}\big(\min(N,N_3)\big)^\frac{n-1}2\normfg.
\end{align*}
Hence we get \eqref{claim}.

We now consider two cases   $1=N_1\sim N_2\gg  N_3,$ $1=N_3\gg N_1\sim  N_2$, separately.
When $1=N_1\sim N_2\gg  N_3$,
from
\eqref{decomp}, triangle inequality and \eqref{claim} we get
\begin{align*}
\|\calh\|_{\mix{\widetilde q'}{\widetilde r'}}
&\lesssim  \sum_{N\le 4} N^{2-n} \|\psi( |\nabla|/{N})(\schr f\schr g)\schr h\|_{\mix{\widetilde q'}{\widetilde r'}}
\\
&\lesssim \sum_{N\le 4}N^{\frac{2-n}2}\big(\min(N,N_3)\big)^\frac{n-1}2\normfg.
\end{align*}Summation in $N$ gives
 \[\|\calh\|_{\mix{\widetilde q'}{\widetilde r'}}\lesssim  N^\frac12_3 \normfg.\]
This proves the case $(i) \; 1=N_1\sim N_2\gg  N_3$.  When $1=N_3\gg N_1\sim  N_2$ note that the summation is taken over $N\lesssim N_1$.  By \eqref{decomp}, triangle inequality and \eqref{claim} we get
 \begin{align*}
\|\calh\|_{\mix{\widetilde q'}{\widetilde r'}}
&\lesssim  \sum_{N\lesssim N_1}N^\frac12\normfg.
\end{align*}
The desired estimate follow from summation in $N$.  This completes the proof of Proposition \ref{har-key-prop}.

\smallskip

Before closing this subsection we state a slightly  strengthened version of Proposition \ref{har-key-prop}  which  is to be used in Section 4.

\begin{cor}\label{key-cor} Let $i=1,2,3.$ If $N_i\sim \min(N_1, N_2, N_3)$, then Proposition \ref{har-key-prop} remains valid  with $A(N_i)$ replaced by $B(0,N_i)$ in the assumption. 
\end{cor}

\begin{proof} When all of $N_1,N_2,N_3\sim \min(N_1, N_2, N_3)$,  \eqref{hart-key} trivially holds  by Lemma \ref{stricha}. If only one $N_i$ of $N_1,N_2,N_3\sim \min(N_1, N_2, N_3)$, then by decomposition of $B(0,N_i)$ into dyadic shells, applying Proposition \ref{har-key-prop} to each dyadic shell and direct summation of geometric series one can easily see that    \eqref{hart-key} holds.  The other possibility is that two of $N_1,N_2,N_3\sim  \min(N_1, N_2, N_3)$. In this case we only need to consider two cases $N_2\sim N_3\ll N_1$ and $N_1\sim N_2\ll N_3$ by symmetry between $N_1$ and $N_2$. For both cases  one can see without difficulty that the argument for the proof of Proposition \ref{har-key-prop}
works for either  $\supp\widehat g\subset B(0,N_2)$ and $\supp\widehat h\subset B(0,N_3)$ or $\supp\widehat f\subset B(0,N_1)$ and $\supp\widehat g\subset B(0,N_2)$.
\end{proof}

We now prove Theorem \ref{trid} by showing \eqref{trid-1}, \eqref{trid-2}, separately.

\subsection{Proof of \eqref{trid-1}}
 For simplicity we denote by $f_{N_j} (j = 1,2,3)$ the Littlewood-Paley projection $P_{N_j} f$ of $f$ . Then we decompose
\begin{align*}
\calhd&=\sum_{N_1,N_2,N_3} \tfgdijk \\
&=\sum_{N_1\le N_2}\sum_{N_3} \tfgdijk +\sum_{N_1> N_2}\sum_{N_3} \tfgdijk.\end{align*}
By symmetry it is enough to handle the first one because the second can be handled similarly. Then we have three possible cases; $N_3\le N_1\le N_2$,  $N_1\le N_3\le N_2$ and $N_1\le N_2\le N_3$. We separately treat the summation of  each case.

\smallskip

\noindent{\textbf{$\text{Case}\; N_3\le N_1\le N_2$.} } This case is the easiest. It can be handled by using the Strichartz estimates only. We claim that
for any positive $\sss$ with $\ssum$,
\begin{equation}
\|\sum_{N_3\le N_1\le N_2}\tfgdijk\amixednorm \lesssim \hnormfg.
\end{equation}
By setting $N_2=N_3N_4\equiv N_{34}$ we write
\begin{align*}
\sum_{N_3\le N_1\le N_2}&\tfgdijk = \sum_{N_4 \ge 1}\sum_{N_3=-\infty}^\infty\sum_{N_3\le N_1\le N_{34}} \tfgdij{1}{34}{3}.
\end{align*}

Using Lemma \ref{stricha}  we get
\begin{align*}
\|\tfgdij{1}{34}{3}\amixednorm
\lesssim N_3\|f_{N_1}\|_{L^2}\|g_{34}\|_{L^2}\|h_3\|_{L^2}.
\end{align*}
So, the norm $\|\sum_{N_3\le N_1\le N_2}\tfgdijk \amixednorm$ is bounded by
\[ C\sum_{N_4\ge 1}\sum_{N_3}\sum_{N_3\le N_1\le N_{34}}
N_3N_1^{-s_1}N_{34}^{-s_2}N_3^{-s_3}(\sfa{1})(\sga{34})(\sha{3}).\]
It is also bounded again by
\begin{align*}
 C\|f\hhilnorm{s_1} \sum_{N_4\ge 1}N_4^{-s_2}
 \sum_{N_3}
(\sga{34})(\sha{3}).
\end{align*}
Then Cauchy-Schwarz inequality yields the desired bound.

\smallskip

\noindent{\textbf{Case $ N_1\le N_3\le N_2$.}}
In this case we set $N_2=N_1N_4 \equiv N_{14}$ and write
\begin{align*}
\sum_{N_1\le N_3\le N_2}&\tfgdijk = \sum_{N_4 \ge 1}\sum_{N_1}\sum_{N_1\le N_3\le N_{14}} \tfgdij{1}{14}{3}.
\end{align*}
Using triangle inequality and Proposition \ref{har-key-prop} we see
\begin{align*}
\|\tfgdij{1}{14}{3}\amixednorm
\lesssim  N_3N_4^{-\frac12}\|f_1\|_{L^2}\|g_{14}\|_{L^2}\|h_3\|_{L^2}.
\end{align*}
Hence the norm $\|\sum_{N_1\le N_3\le N_2}\tfgdijk\amixednorm$ is bounded by
\begin{align*}
\|h\hhilnorm{s_3}\sum_{N_4 \ge 1} N_4^{-\frac12}
\sum_{N_1}\sum_{N_1\le N_3\le N_{14}}N_3^{1-s_3}
N_1^{-s_1}N_{14}^{-s_2}(\sfa{1})(\sga{14}).
\end{align*}
Taking summation in $N_3$ and using Schwarz's inequality in $N_1$, we bounds this by
\[ C\|f\hhilnorm{s_1}\|g\hhilnorm{s_2}\|h\hhilnorm{s_3} \sum_{N_4 \ge 1} N_4^{s_1-\frac12}.\]
Note that $s_1<\frac12$ because $s_3>\frac12$. Hence  we get the desired.

\smallskip

\noindent{\textbf{$\text{Case } N_1\le N_2\le N_3$.}}
We set $N_3=N_1N_4 \equiv N_{14}$ and write
\begin{align*}
\sum_{N_1\le N_2\le N_3}&\tfgdijk = \sum_{N_4 \ge 1}\sum_{N_1}\sum_{N_1\le N_2\le N_{14}} \tfgdij{1}{2}{14}.
\end{align*}
Using triangle inequality and Proposition \ref{har-key-prop} we have
\begin{align*}
\|\sum_{N_1\le N_2\le N_3}&\tfgdijk\amixednorm\\
&\lesssim \sum_{N_4 \ge 1}\sum_{N_1}\sum_{N_1\le N_2\le N_{14}}  N_{14}N_4^{-\frac12}\|f_1\|_{L^2}\|g_{2}\|_{L^2}\|h_{14}\|_{L^2}.
\end{align*}
The left hand side of the above is bounded by
\begin{align*}
C\|g\hhilnorm{s_2}\sum_{N_4 \ge 1} N_4^{-\frac12}
\sum_{N_1}\sum_{N_1\le N_2\le N_{14}}
N_1^{-s_1}N_2^{-s_2}N_{14}^{1-s_3}(\sfa{1})(\sha{14}).
\end{align*}
Taking summation in $N_2$ and using Schwarz's inequality in $N_1$, the above is bounded  by a constant multiple of
$\|f\hhilnorm{s_1}\|g\hhilnorm{s_2}\|h\hhilnorm{s_3} \sum_{N_4 \ge 1} N_4^{-s_3+\frac12}.$
Since $s_3>\frac12,$  we get the desired.

\subsection{Proof of \eqref{trid-2}}
We decompose
\begin{align*}
\caldh=\sum_{N_1,N_2,N_3} \tddfgijk.
\end{align*}
There is no obvious symmetry. We should consider the following six cases:
\begin{eqnarray*}
      &\,\,(i)\, N_1\ge  N_2\ge  N_3,\;  \,\,\,(ii)\, N_1\ge N_3\ge N_2, \; \,(iii)\, N_3\ge N_1\ge  N_2,\\
      &(iv)\, N_2\ge N_1\ge N_3,  \; (v)\, N_2\ge N_3\ge  N_1, \; \,(vi)\, N_3\ge N_2\ge N_1.
\end{eqnarray*}
As expected, the cases $(i)$ and $(ii)$ are the major parts. The others are sort of minor terms. Each of the cases can be handled by the same argument as before.

\noindent{\textbf{$\text{Case }(i) \ N_1\ge N_2\ge N_3$:}} To begin with,
we set $N_1=N_{34}$. By the triangle inequality and rearrangement of the summation we get
\begin{align*}
\|\sum_{N_3\le N_2\le N_1}&\tddfgijk\amixednorm\\
&\lesssim
\sum_{N_4 \ge 1}\sum_{N_1}\sum_{N_{34}\ge N_2\ge N_{3}}\|\tddfgij{34}{2}{3}\amixednorm.
\end{align*}
Applying Proposition \ref{har-key-prop}
we bound the left hand side of the above by
\begin{align*}
\|g\hhilnorm{s_2}\sum_{N_4 \ge 1} N_4^{-\frac12}
\sum_{N_3}\sum_{N_{34}\ge N_2\ge N_{3}}
N_{34}^{1-s_1}N_2^{-s_2}N_{3}^{-s_3}(\sfa{34})(\sha{3}).
\end{align*}
Taking summ in $N_2$ and using Schwarz's inequality in $N_1$,  we see that
\begin{align*}
\|\sum_{N_3\le N_2\le N_1}\tddfgijk\amixednorm
&\lesssim
\|f\hhilnorm{s_1}\|g\hhilnorm{s_2}\|h\hhilnorm{s_3} \sum_{N_4 \ge 1} N_4^{-s_1+\frac12}.
\end{align*}
Since $s_1>\frac12,$  we get the desired.

\smallskip

\noindent{\textbf{$\text{Case }(ii) \ N_1\ge N_3\ge N_2$:}}
We set $N_1=N_{24}$ and rearrange the summation such that
\[\sum_{\ N_1\ge N_3\ge N_2}=\sum_{N_1}\sum_{N_2}\sum_{N_{24}\ge N_3\ge N_2}.\]
While applying triangle inequality and Proposition \ref{har-key-prop}, the only difference to the previous case $(i)$ is that $N_3$ is replaced by $N_2$.  Hence by the same argument we get
\begin{align*}
\|\sum_{\ N_1\ge N_3\ge N_2} (\cdot)\amixednorm\lesssim
\|f\hhilnorm{s_1}\|g\hhilnorm{s_2}\|h\hhilnorm{s_3} \sum_{N_4 \ge 1} N_4^{-s_1+\frac12}.
\end{align*}
Since $s_1>\frac12,$ we get the desired.

\smallskip
The remaining cases $(iii)-(vi)$ can be handled by the same way. Repeating the argument one can show
\begin{align*}
\|\sum_{N_3\ge N_1\ge N_2} (\cdot)\amixednorm, \ \lesssim
\|f\hhilnorm{s_1}\|g\hhilnorm{s_2}\|h\hhilnorm{s_3} \sum_{N_4 \ge 1} N_4^{-\frac12+s_2},\\
\|\sum_{N_2\ge N_1\ge N_3} (\cdot)\amixednorm \lesssim \|f\hhilnorm{s_1}\|g\hhilnorm{s_2}\|h\hhilnorm{s_3} \sum_{N_4 \ge 1} N_4^{-\frac12+s_3}.
\end{align*}
Since $s_1>\frac12,$ $s_2, s_3<\frac12$, we get the desired. One can also show
\begin{align*}
\|\sum_{N_3\ge N_2\ge N_1} (\cdot)\amixednorm, \ \lesssim
\|f\hhilnorm{s_1}\|g\hhilnorm{s_2}\|h\hhilnorm{s_3} \sum_{N_4 \ge 1} N_4^{-\frac12-s_3},\\
\|\sum_{N_2\ge N_3\ge N_1} (\cdot)\amixednorm \lesssim \|f\hhilnorm{s_1}\|g\hhilnorm{s_2}\|h\hhilnorm{s_3} \sum_{N_4 \ge 1} N_4^{-\frac12-s_2}.
\end{align*}
Therefore this completes the proof of Theorem \ref{trid}.

\section{Smoothing properties}
In this section we prove Theorem \ref{hart}. The proof relies on the arguments using the Bourgain space $\mathbf{X}^{s,\,b}$  for $s, b \in \bbr$. It consists of the functions $u$ such that
\[\|u\|_{\mathbf X^{s,\,b}} \equiv \left(\int\!\!\int \langle \xi\rangle^{2s}\langle \tau-|\xi|^2\rangle^{2b}|\widetilde{u}(\tau, \xi)|^2\,d\tau d\xi\right)^\frac12 < \infty,\]
where $\widetilde{u}(\tau, \xi)$ is the time-space Fourier transform of $u$. We also use the norm $\mathbf{X}^{s,\,b}(J_T)$ for time interval $J_T = [0,T]$ defined as
\[\|u\|_{\mathbf X^{s,\,b}(J_T)} \equiv \inf \{\|\varphi\|_{\mathbf X^{s,\,b}} : \varphi|_{J_T} = u\}.\]

\subsection{ Hartree type nonlinearity}
In this section $V(u)$ denotes $\kappa |x|^{-2}\ast |u|^2$.
Let us invoke $V(u)=c|\nabla|^{2-n}(|u|^2)$ for some constant $c$.
Using  the $\xsp$ spaces and Theorem \ref{trid},
  one can derive the following.

\begin{prop}\label{trixxx} Let $n\ge 3$. Then for any $s, b >\frac12$ there exists $0 < \epsilon \ll 1$ such that
\begin{equation} \label{xx}
\|[V(u)u] \xsnorm{1}{-\frac12+\epsilon}
\lesssim \|u\xsnorm{s}{b}^3.
\end{equation}
\end{prop}
\begin{proof} We first show that
for  $s_0 > 2 + \frac n2$
\begin{align}\label{tri-infty}
\|\caldh\|_{\mix\infty\infty} + \|\calhd\|_{\mix\infty\infty}
\lesssim\|f\hilnorm{s_0}\|g\hilnorm{s_0}\|h\hilnorm{s_0}.\end{align}
Here we do not intend to obtain sharp $s_0$ but we here are content with some crude estimate
which is enough for our purpose. From the Sobolev embedding we note that
$\|\schr \langle \nabla \rangle f\|_{\mix\infty\infty} \lesssim \|f\hilnorm{s_0}$.
Hence it is enough to show that
\[\||\nabla|^{2-n}(\schr \langle\nabla \rangle f\schr g)\|_{\mix\infty\infty}\lesssim \|f\hilnorm{s_0}\|g\hilnorm{s_0}.\]
But this follows easily from the observation that
\begin{align*}
\||\nabla|^{2-n}G\|_{L_x^\infty}
&\lesssim \bigg(\int_{|\xi|\le 1} + \int_{|\xi|\ge 1} \bigg)|\xi|^{2-n} |\widehat G(\xi)| d\xi
\lesssim \|G\|_{L^1} + \|G\|_{H^1}
 \end{align*}
together with Leibniz rule and Schwarz's inequality.

\smallskip

\smallskip

Now interpolating \eqref{tri-infty} with the estimates \eqref{trid-1}  of Theorem \ref{trid}, we see that
if $(\tilde q, \tilde r)$ is any admissible pair, $s_3> \frac12$ and $\sum_{i=1}^3 s_i > 1$, then there exist $0 < \ep_1, \ep_2 \ll 1$ such that
\begin{align}\label{trid3}
\|\calhd\|_{\mix{\tilde q'+\epsilon_1}{\tilde r'+\epsilon_2}}\lesssim  \|f\|_{H^{s_1}}\|g\|_{H^{s_2}}\|h\|_{H^{s_3}}.
\end{align}
Similarly, using \eqref{tri-infty} and \eqref{trid-2}, for any admissible pair $(\tilde q, \tilde r)$ and the exponents $s_1, s_2, s_3$ with $s_1 > \frac12, \sum_{i=1}^3 s_i > 1$, we can find $\ep_1$ and $\ep_2$ such that
\begin{align}\label{trid4}
\|\caldh\|_{\mix{\tilde q'+\epsilon_1}{\tilde r'+\epsilon_2}}\lesssim \|f\|_{H^{s_1}}\|g\|_{H^{s_2}}\|h\|_{H^{s_3}}.
\end{align}

One may  write
$u(t,x)=c_n
\int  e^{it\tau}
\int_{\mathbb R^{n}} e^{i(x\cdot\xi-t|\xi|^2)} \widetilde u(\tau-|\xi|^2, \xi) d\xi d\tau$
by inversion and translation in frequency variables.
Hence, we get
\[(|\nabla|^{2-n}(u v) \nabla w)(x,t)=\iiint \mathcal H (f_\tau, g_{\tau'}, \nabla h_{\tau''}) (x)\,d\tau d\tau'd\tau'',\]where
$\widehat f_\tau(\xi)=\widetilde  u(\tau-|\xi|^2, \xi),\, \widehat g_\tau(\xi)=\widetilde  v(\tau-|\xi|^2, \xi), \, h_\tau(\xi) = \widetilde  w(\tau-|\xi|^2, \xi).$
From \eqref{trid3} and Minkowski's inequality  it follows  that
\[\||\nabla |^{2-n}(u v) \nabla w\|_{\mix{\tilde q'+\epsilon_1}{\tilde r'+\epsilon_2}} \lesssim \iiint \|f_\tau\|_{ H^{s_1}} \|g_{\tau'}\|_{H^{s_2}}\|h_{\tau''}\|_{ H^{s_3}}\, d\tau d\tau'd\tau''.\]
Plancherel's theorem and Schwarz's inequality yield
\begin{align}\label{trid5}
&\||\nabla|^{2-n}(uv) \nabla w\|_{\mix{\tilde q'+\epsilon_1}{\tilde r'+\epsilon_2}} \lesssim \xsnormuv
\end{align}
for any $b > \frac12$ and for any $(\tilde q, \tilde r)$, $s_1, s_2, s_3$ as in \eqref{trid3}.
By repeating the same argument with \eqref{trid4}, we also get
\begin{align}\label{trid6}
&\||\nabla|^{2-n}(\nabla u v) w\|_{\mix{\tilde q'+\epsilon_1}{\tilde r'+\epsilon_2}} \lesssim \xsnormuv
\end{align}
for any $b > \frac12$ and for any $(\tilde q, \tilde r)$, $s_1, s_2, s_3$ as in \eqref{trid3}.

We now fix $s, b>1/2$. To show \eqref{xx}, we need to show that
\[\|V(u)u\|_{\mathbf X^{0, -\frac12+\epsilon}},\, \|\nabla [V(u)u]\|_{\mathbf X^{0, -\frac12+\epsilon}} \lesssim\|u\|_{\mathbf X^{s, b}}^3 .\] By duality  it suffices to  show that
$$ |\inp{\psi}{U}|
\lesssim \|\psi\|_{\mathbf X^{0, \frac12-\ep}} \|u\xsnorm{s}{b}^3$$
for $U=|\nabla|^{2-n}(|u|^2)u,$ $|\nabla|^{2-n}(\nabla u \overline u)u,$ $|\nabla|^{2-n}( u \nabla \overline u)u$ and $|\nabla|^{2-n}( |u|^2)\nabla u$.
We first handle the case $U=|\nabla|^{2-n}(\nabla u \overline u)u$. By H\"{o}lder's inequality we have
$$
|\inp{\psi}{|\nabla|^{2-n}(\nabla u \overline u)u}| \lesssim \|\psi\|_{L_t^{\tilde q}L_x^{\tilde r}}\||\nabla|^{2-n}(\nabla u \overline u)u\|_{L_t^{\tilde q'}L_x^{\tilde r'}}.
$$
By \eqref{trid6}  we can choose $(1/\tilde q,1/\tilde r)$ close enough to the Strichartz line $2/q + n/r = n/2$ so that $2/\tilde q + n/\tilde r > n/2$ and $\||\nabla|^{2-n}(\nabla u \overline u)u\|_{L_t^{\tilde q'}L_x^{\tilde r'}}\lesssim   \|u\xsnorm{s}{b}^3$.  By the choice of $(\tilde q,\tilde r)$ and Lemma \eqref{xsb-str}, we see $\|\psi\|_{L_t^{\tilde q}L_x^{\tilde r}}\lesssim \|\psi\|_{\mathbf X^{0, \frac12-\ep}} $ for some $\epsilon>0$. Hence we get the desired. The  remaining cases $U=|\nabla|^{2-n}( u \nabla \overline u)u$, $|\nabla|^{2-n}( |u|^2)\nabla u$, $|\nabla|^{2-n}(|u|^2)u$ can be similarly  shown using \eqref{trid6}, \eqref{trid5} and Lemma \ref{xsb-str}.
  This completes the proof of Proposition \ref{trixxx}.
\end{proof}

\begin{lem}\label{xsb-str}
Let  $q,r \ge 2$ (with possible exception when $n=2$).
If  $\frac 2{q}+\frac n{r}\le \frac n2$, then
 \[ \|u\|_{\mix qr}  \lesssim\|u\xsnorm{s}{b}\]
for $s=s(q,r)=\frac n2-\frac 2{q}-\frac n{r}$ and any $b > \frac12$. If $\frac 2{q}+\frac n{r} > \frac n2$,
\[ \|u\|_{\mix qr} \lesssim \|u\xsnorm{0}{b+ \epsilon}\]
for $b=b(q,r)=\frac {n+2}4-\frac1q-\frac n{2r}$ and any $\epsilon>0$.
\end{lem}

\noindent The first  follows from the estimates
$\|\schr f\|_{\mix qr} \lesssim\|f\hhilnorm{s(q,r)}$ and the standard argument (for instance see \cite{tb}). Interpolation between the first estimate and the trivial $\|u\|_{\mix{2}{2}}\le \|u\xsnorm{0}{0}$ give the second.

\smallskip

By using the Proposition \ref{trixxx} and standard fixed point argument in $\mathbf X^{1,\,\frac12+\ep}$ for $0< \ep \ll1$, we prove the first part of Theorem \ref{hart}. Here we note that $\mathbf{X}^{1,\frac12+\epsilon}(J_T) \hookrightarrow C([0,T]; H^{1} (\bbr ^n))$.

\begin{proof}[Proof of $(1)$ of Theorem \ref{hart}]
We first show the local well-posedness. For this purpose we define a nonlinear functional $\mathcal N$ by
$$
\mathcal N(u) = \phi(t)e^{i\Delta}u_0 - i\phi(t/T)\int_0^t e^{i(t-t')\Delta} [V(u(t'))u(t')]\,dt',
$$
where $\phi$ is a fixed smooth cut-off function such that $\phi(t) = 1$ if $|t| < 1$ and $\phi(t) = 0$ if $|t| > 2$, and $0 < T \le 1$ is fixed.
Then we use the well-known properties of $\mathbf{X}^{s,b}$ (for instance see Proposition 2.2 of \cite{kv});
\begin{align}
\|\phi(t) e^{it\Delta}u_0\|_{\mathbf{X}^{s,b}} \lesssim \|u_0\|_{H^s}
\label{homo xsb}
\end{align}
 for any $s,b$, and
\begin{align}
\|\int_0^t e^{i(t-t')\Delta} F(t',x)\,dt'\|_{\mathbf X^{s,b}(J_T)} \lesssim T^{1+b'-b}\|F\|_{\mathbf X^{s,b'}(J_T)},\label{inhomo xsb}
\end{align}
for $s \in \mathbb{R}$ and $-\frac12 < b' \le 0, 0 \le b \le b'+1$.

Let us define a complete metric space $B_{T,\rho}$ by $$ B_{T,\,\rho} = \{u \in \mathbf X^{s, \frac12+\epsilon}(J_T): \|u\|_{\mathbf X^{s,\frac12+\epsilon}(J_T)} \le \rho\}$$ with metric $d$ such that $d(u, v) = \|u - v\|_{\mathbf X^{s, \frac12+\epsilon}(J_T)}$.
From  \eqref{homo xsb} and \eqref{inhomo xsb} with $b = \frac12 + \epsilon, b' = -\frac12+\epsilon', \epsilon < \epsilon'$ it follows that for any $u \in B_{T,\rho}$
$$
\|\mathcal N(u)\|_{\mathbf X^{s, \frac12+\epsilon}} \lesssim \|u_0\|_{H^s} + T^{\epsilon'-\epsilon}\|V(u)u\|_{\mathbf X^{1, -\frac12+\epsilon'}}.
$$
If $\epsilon'$ is sufficiently small, then we deduce from Proposition \ref{trixxx} that
\begin{align*}
\|\mathcal N(u)\|_{\mathbf X^{s, \frac12+\epsilon}} \lesssim \|u_0\|_{H^s} + T^{\epsilon'-\epsilon}\|u\|_{\mathbf X^{s, \frac12+\epsilon}(J_T)}^3 \lesssim  \|u_0\|_{H^s} + T^{\epsilon'-\epsilon}\rho^3.
\end{align*}
Choosing $\rho$ and $T$ such that $\rho \ge 2C\|u_0\|_{H^s}$ and $CT^{\epsilon'-\epsilon}\rho^3 \le \rho/2$ for some constant $C$, we see that
the functional $\mathcal N$ is a map from $B_{T,\rho}$ to itself.
One can now easily observe that $\mathcal N$ is a contraction. In fact, using Proposition \ref{trixxx} again, one can easily see that for any $u, v \in B_{T,\rho}$ and for sufficiently small $T$
\begin{align*}
d(\mathcal N(u), \mathcal N(v)) &\lesssim T^{\epsilon'-\epsilon}\|V(u)u- V(v)v\|_{\mathbf X^{s, -\frac12 + \epsilon'}(J_T)}\\
 &\lesssim T^{\epsilon'-\epsilon}(\|u\|_{\mathbf X^{s, \frac12+\epsilon}(J_T)}+\|v\|_{\mathbf X^{s, \frac12+\epsilon}(J_T)})^2d(u, v)\\
 & \lesssim T^{\epsilon'-\epsilon}\rho^2d(u,v).
\end{align*}
Hence a choice of small $T$ makes $\mathcal N$ be a contraction map. Therefore there is a unique $u \in \mathbf X^{s, \frac12+\epsilon}(J_T)$ such that
$u(t) = e^{it\Delta} u_0 + D(t)$, where
$$
D(t) = -i\int_0^t e^{i(t-t')\Delta}[V(u)u(t')]\,dt'.
$$
In view of Proposition \ref{trixxx} and the estimate \eqref{inhomo xsb}, we have for $s>1/2$
\begin{align*}
\|D\|_{\mathbf X^{1, \frac12+\epsilon}(J_T)} \lesssim \|V(u)u\|_{\mathbf X^{1, -\frac12+\epsilon}(J_T)} \lesssim \|u\|_{\mathbf X^{s, \frac12+\epsilon}(J_T)} < \infty.
\end{align*}
Hence the smoothing effect is obtained.
\end{proof}

\subsection{ Power type nonlinearity}

Adopting the argument in the proof of Proposition \ref{trixxx} and using Theorem \ref{hilbert}, one can easily get
the following.
\begin{cor}\label{x-estimate} Let $(q,r)$ be a admissible pair satisfying that $2 < r \le 4$ when $n= 3$ and $q, r >2$ when $n \ge 4$.
Then  for $b>\frac12$ and every $0<s<1-2/r$ there holds
\[ \|u \nabla v\mixednorm \lesssim \|u\xsnorm{s}{b} \|v \xsnorm{1-s}{b}.\]
\end{cor}

The local well-posedness of the Cauchy problem \eqref{hartree} with $V(u) = \kappa |u|^\frac4n$ is well-known in $H^s$ space and also in $\mathbf X^{s,\, b}$ space \cite{kv}. Hence using Corollary \ref{x-estimate} and following the lines of argument in \cite{kv} we get the proposition.
\begin{prop}\label{nonlinear} Let $n\ge 3$, $s_n = \frac12$ if $n = 3$, and $s_n=1-(\frac{4}{n})(\frac{2} n)$ if $n \ge 4$.   Then  for $b>\frac12$ and every $s>s_n$ there is an $\epsilon>0$ such that
\[\|\nabla (|u|^\frac4n u)\xsnorm{0}{-\frac12+\epsilon} \lesssim\|u\xsnorm{s}{b}^{\frac 4n+1}.\]
\end{prop}

\noindent Once this is established, the proof of the second part of Theorem \ref{hart} is almost same with the case of Hartree type nonlinearity, part (1). Hence we omit the detail.

\begin{proof}[Proof of Proposition \ref{nonlinear}]
Using duality it is enough to show that
\[|\inp{\psi}{\nabla (|u|^\frac4n u)}|\lesssim \|\psi\xsnorm{0}{\frac12- \epsilon}\|u\xsnorm{s}{b}^{\frac 4n+1}. \]
By direct differentiation the left hand side is bounded by a constant multiple of
\[\inp{|\psi|}{|u|^\frac4n|\nabla u|}\lesssim \sum_{N \ge 1} \inp{|\psi|}{|u|^\frac4n|\nabla u_N|}.\]
Here $u_N = P_N u$ for dyadic $N > 1$ and $u_1 = \widetilde P_1 u$ for the projection operator $\widetilde P_1$ (recall the notation in introduction).
Using H\"older's inequality, we see that for $n = 3$
\begin{align*}
\inp{|\psi|}{|u|^\frac4n|\nabla u_N|}
   \lesssim \|\psi\|_{\mix{q_1}{r_1}}
        \|u\nabla u_N\|_{\mix{q/2}{r/2}}
           \|u\|_{\mix{q_2}{r_2}}^{\frac13},
\end{align*}
where $(1/q_1, 1/r_1) + (2/q, 2/r) + (1/3)(1/q_2, 1/r_2)=1$, and for
$n \ge 4$
\begin{align*}
\inp{|\psi|}{|u|^\frac4n|\nabla u_N|}
   \lesssim \|\psi\|_{\mix{q_1}{r_1}}
        \|u\nabla u_N\|_{\mix{q/2}{r/2}}^\frac{4}n
           \|\nabla u_N\|_{\mix{q_2}{r_2}}^{1-\frac{4}n},
\end{align*}
where $(1/q_1, 1/r_1)+ (4/n)(2/q, 2/r) + (1-4/n)(1/q_2, 1/r_2)=1$. We
want to choose admissible $(q,r)$ which is arbitrarily close to $(\frac83, 4)$ for $n =
3$ and  $(2,\frac{2n}{n-2})$ for $n \ge 4$,
respectively, and non-admissible pairs $(q_1,r_1)$ and $(q_2,r_2)$
such that  $(1/q_1,1/r_1)$ and $(1/q_2,1/r_2)$ are slightly above and below the Strichartz line, respectively.
More precisely, $\epsilon_1>2/{q_1}+n/{r_1}-n/2>0$ and $0>  2/{q_2}+n/{r_2}-n/2>-\epsilon_2$ for small $\epsilon_1,\epsilon_2>0$.
 With the choices of $(q_1,r_1)$ and  $(q_2,r_2)$, using Lemma \ref{xsb-str} and Proposition \ref{x-estimate}, we have for $|s_1| < 1 - \frac2r$ and $\epsilon>0$
\begin{align*}
\inp{|\psi|}{|u|^\frac4n|\nabla u_N|}
 &\lesssim \|\psi\xsnorm{0}{\frac12-\epsilon}
   \|u\xsnorm{s_1}{b} \|u_N \xsnorm{1-s_1}{b} \| u\xsnorm{\epsilon_0}{b}^\frac13 \\
    & \lesssim N^{1-s_1-s}\|\psi\xsnorm{0}{\frac12-\epsilon}\|u\xsnorm{s_1}{b}\|u\xsnorm{s}{b}
     \|u\xsnorm{s}{b}^{\frac13}
\end{align*}
 when $n = 3$,  and
\begin{align*}
\inp{|\psi|}{|u|^\frac4n|\nabla u_N|}
      %
        &\lesssim  \|\psi\xsnorm{0}{\frac12-\epsilon}
                \|u\xsnorm{s_1}{b}^\frac4n \|u_N \xsnorm{1-s_1}{b}^\frac4n \|\nabla u_N\xsnorm{\epsilon_0}{b}^{1-\frac4n} \\
                   & \lesssim N^{\left(1-s-s_1 \cdot 4/n + \epsilon_0(1 -4/n)\right)} \|\psi\xsnorm{0}{\frac12-\epsilon}\|u\xsnorm{s_1}{b}^\frac4n\|u\xsnorm{s}{b}^\frac4n \|u\xsnorm{s}{b}^{1-\frac4n}
\end{align*}
when  $n \ge 4$.
Therefore, for $s > \max(s_1,1-s_1)$
we get
\[|\inp{\psi}{|u|^\frac4n u}|\lesssim \sum_{N\ge 1} N^{1-s-s_1} \|\psi\xsnorm{0}{\frac12- \epsilon}\|u\xspnorm^{\frac 43+1}  \]
when $n=2$, and
\[|\inp{\psi}{|u|^\frac4n u}|\lesssim \sum_{N\ge 1} N^{\left(1-s-s_1 \cdot 4/n + \epsilon_0(1 -4/n)\right)} \|\psi\xsnorm{0}{\frac12- \epsilon}\|u\xspnorm^{\frac 4n+1}\]
when $n \ge 4$.
So we get the desired bound provided $s > 1-s_1$ when $n = 3$ and
$s>1-s_1\frac 4n$ when $n \ge 4$.
We now choose admissible pair $(q,r)$ to be arbitrarily close to $(\frac83, 4)$ when $n = 3$ and  $(2,\frac{2n}{n-2})$ when $n \ge 4$. Then we get the desired bound for
$s > \frac12$ when $n = 3$, and for $ s>1-\frac 8{n^2}$ when $n \ge 4$ because we can choose $s_1$ to be arbitrarily close to $\frac12$ when $n = 3$ and to $\frac2n$ when $n \ge 4$, and $1-\frac 8{n^2}> \max(\frac2n,1-\frac2n)$. This completes the proof of Proposition \ref{nonlinear}.
\end{proof}

\section{Global well-posedess of Hartree equations.}

In this section we give the proof Theorem \ref{T} which  improves the global well-posedness results in
\cite{ck}. Based on $I$-method,  the two main ingredients are the almost energy conservation and almost interaction Morawetz inequality. Our improvement results from  the better decay control of
these crucial estimates (Proposition \ref{ACL}, \ref{ebound}), which are obtained by exploiting the trilinear interaction estimate (Proposition \ref{har-key-prop}).
Since we basically follow the usual steps of  I-method (\cite{ck, cgt,dpst}),
we do not intend to give all the details of the proof. Instead, we are devoted to proving Proposition \ref{ACL}, \ref{ebound} after giving a brief explanation about the overall argument .

\smallskip

The $I$-method, introduced by Colliander \emph{et.al.} \cite{ckstt02} to handle low regularity initial data, makes use of a smoothing operator $I$ which regularizes a rough solution up to the regularity level of a conservation law by damping high frequency part. For $0 < s < 1$ the operator $I:H^s \to H^1$ (depending on a parameter $N \gg 1$) is defined by
\begin{equation*}
\widehat{If}(\xi) \equiv m(\xi)\widehat f(\xi),
\end{equation*}
where the multiplier $m(\xi)$ is smooth, radially symmetric,
nonincreasing in $|\xi|$ and satisfies
\begin{equation*}
m(\xi)=
\begin{cases}
\begin{array}{ll}
1 & |\xi|\le N \\
\left( \frac {N}{|\xi|}\right)^{1-s} & |\xi|\ge 2N .
\end{array}
\end{cases}
\end{equation*}
When the solution $u$ of \eqref{hartree} is in $H^s$, $0<s<1$, $E(u)$ may not be finite, but
 $E(Iu)$ is finite. Since $Iu$ is not a solution to \eqref{hartree}, $E(Iu)$ is not expected to be  conserved. However, it is \emph{almost  conserved} and  the deviation can be controlled by $O(N^{-\sigma})$ , $\sigma>0$ since the operator $I$ gets close to the identity as $N$ increases.
 In Proposition \ref{ACL} we show that  for $ p= 3/2$
  \begin{align}\label{egrow}
  E(Iu)(T) = E(Iu_0)+ N^{-p+}\Gamma(Z_I(T)),
  \end{align}
where $\Gamma(r) = \sum_{1 \le i \le k}O(r^{m_i})$ for some $k,m_1,\dots, m_k \ge 1$ and $Z_I(T)$ is the iteration space norm defined by
$$
Z_I(T) = \sup_{(q,r)\,\text{:\,admissible}}\|I \lnabla u\|_{L_t^qL_x^r(J_T\times \bbr^n)}.
$$
 After the Morawetz interaction potential for 3-d NLS  was introduced by Colliander \emph{et al.} \cite{ckstt04}, it was extended to other dimension
\cite{cgt,dpst,tvz}.  To make use of such estimates (e.g. local in time Morawetz inequality in \cite{dpst}),
the restriction $s>1/2$ is inevitable. However, this restriction can be removed by using an inequality for  $Iu$ (\cite{cgt}). In fact, it is almost valid in the sense that for some $\theta, p > 0$
\begin{equation}\label{Imo}
\|Iu\|_{\pair(J_T\times \bbr^n)}\lesssim T^{\theta}
  \|Iu_0\|_{L^2_x}^{\frac 12}\|Iu\|^{\frac{n-2}{n-1}}_{L_t^{\infty}\dot{H}^{\frac 12}_x (J_T\times \bbr^n)} + N^{-p+}\Gamma(Z_I(T)).
\end{equation}
The iteration norm $Z_I(T)$ is controlled initially provided that the critical Strichartz norm of
$Iu$ is small (Lemma $3.1$ in \cite{ck}). More precisely,  there is $\delta>0$ such that if
\[ \|Iu\|_{\pair(J_T\times \bbr^n)} \le \delta,  \mbox{ then } Z_I(T)\lesssim \|\lnabla Iu_0\|_{L^2}(\bbr^n).\]
Therefore, once \eqref{egrow}, \eqref{Imo} are obtained, the global well-posedness follows from the usual accounting argument (see Section 5 in \cite{ck} or \cite{cgt}
for details). The threshold regularity $s$ is determined by the decay rate $N^{-p+}$.
Going over the argument in \cite{ck}, one  gets the global well-posedness for $u_0\in H^s$ as long as  $N \gg 1$ can be chosen such that
\[
 KN^{\frac{1-s }{s }\frac{2(n-2) }{n}} \sim N^{p-}
\]
 for any arbitrarily  large $K$. This is possible if $s >  \frac{2(n-2) }{(2+p)n-4 }$. Consequently, by  Proposition \ref{ACL}, \ref{ebound}  we conclude Theorem \ref{T} with $p=3/2$.

\smallskip

The rest of this section
is devoted to the proofs of the almost energy conservation
and almost interaction Morawetz inequality; Proposition \ref{ACL}, \ref{ebound}.

\begin{prop}\label{ACL} Let  $0 < s < 1$ and $N\gg 1$.
Suppose that  $u_0 \in
C_0^{\infty}(\bbr^n)$ and $u$ is the solution to \eqref{hartree}. 
Then for any $\ep > 0$ and $T > 0$
\[| E(Iu)(T) - E(Iu_0) |\lesssim N^{-{\frac 32+\epsilon}}(Z_I(T)^4 + Z_I(T)^6 + Z_I(T)^{10} + Z_I(T)^{12}).\]
\end{prop}
Now let us consider the $I$-Hartree equation by
\begin{align*}
 i(Iu)_t& + \Delta (Iu) = V(Iu)Iu + \big[I(V(u)u)  - V(Iu)Iu \big] \equiv \mathcal{N}_{good} + \mathcal{N}_{bad},
\end{align*}
where $V(u) = |x|^{-2}\ast |u|^2$.
Similarly to the formula $(4.31)$ in \cite{ck}, we have
\begin{align}\label{Imora}
\begin{aligned}
 - (n-1)\int_0^T &\iint_{\bbr^n\times\bbr^n} \Delta\big(\frac {1}{|y-x|}\big)|Iu(x,t)|^2|Iu(y,t)|^2dxdydt \\
& + 2\int_0^T\iint_{\bbr^n\times\bbr^n}|Iu(x,t)|^2\frac{y-x}{|y-x|}\cdot\{\mathcal{N}_{good},Iu\}(y,t) dxdydt \\
& + 2\int_0^T\iint_{\bbr^n\times\bbr^n}|Iu(x,t)|^2\frac{y-x}{|y-x|}\cdot\{\mathcal{N}_{bad},Iu\}(y,t) dxdydt \\
& \lesssim \|Iu\|^2_{L^\infty_tL^2_x (J_T\times
\bbr^n)}\|Iu\|^2_{L^\infty_t\dot{H}^{\frac 12}(J_T\times \bbr^n)}.
\end{aligned}
\end{align}
Here $\{f, g\}$ denotes $\text{Re}(f\nabla \overline g - g \nabla \overline f)$.
Since the second term of \eqref{Imora} is positive, it follows from the H\"{o}lder's inequality and interpolation (for details see Proposition $4.1$ in \cite{ck} or Lemma $5.6$ in \cite{tvz}) that
\begin{align}\label{imora}
\|Iu\|_{\pair(J_T\times \bbr^n)}\lesssim T^{\frac{n-2}{4(n-1)}}
  \|Iu_0\|_{L^2_x}^{\frac 12}\|Iu\|^{\frac{n-2}{n-1}}_{L_t^{\infty}\dot{H}^{\frac 12}_x (J_T\times \bbr^n)} + \mbox{ Error },
\end{align}
where
\[ \text{Error} = \left|\int_0^T\iint_{\bbr^n\times\bbr^n}|Iu(x,t)|^2\frac{y-x}{|y-x|}\cdot\{\mathcal{N}_{bad},Iu\}(y,t) dxdydt\right| .\]
\begin{prop}\label{ebound}
  Let  $0 < s < 1$ and $N\gg 1$.
Suppose that  $u_0 \in
C_0^{\infty}(\bbr^n)$ and $u$ is the solution to \eqref{hartree}. 
Then  for any  $\ep > 0$ and $T > 0$.
\[ {\rm Error } \lesssim N^{-\frac 32 + \ep}(Z_I(T)^6 + Z_I(T)^{12}).\]
\end{prop}

In the whole argument $N$ is assumed to be sufficiently large. So the small frequency part of the solution does not play any significant role. Hence we do not need dyadic decomposition for such portion.   Here, we recall that $\widetilde P_1=id -\sum_{N>1}P_N$. For simplicity,
abusing notation,  we keep denoting  $\widetilde P_1$ by $P_1$.  Throughout  this section  $N_1, \dots, N_4$ are dyadic numbers $\ge 1$ and $\sum_{N_j\ge 1}P_{N_j}=id$ for $j=1,\dots, 4$.

\subsection{Preliminary estimates}

We first show the following inhomogeneous estimate (cf.
\cite{ckstt}) for the solutions with localized frequency.  For the
simplicity of notations  we also denote $ f_j=P_{N_j} f$, $
j=1,2,3$.

  \begin{lem}\label{4.2} Let $N_1, N_2,N_3$ be  dyadic numbers $\ge 1$ and let $u$ be a smooth solution of $iu_t + \Delta u = F$ on $J_T \times \bbr^n$ with the initial data $u_0$.  Then for   $(u_1, u_2, u_3)$ = $(P_{N_1}u, P_{N_2}u, P_{N_3}u)$ it holds that for any admissible pair $(\tilde q, \tilde r)$
\begin{align*}
\sup_{(\eta_1, \eta_2, \eta_3) \in \bbr^{3n}}\||\nabla|^{2-n}(u_1(\cdot-\eta_1)&u_2(\cdot-\eta_2))u_3(\cdot-\eta_3)\|_{\amixedI}\\
&\lesssim C(N_1, N_2,N_3)(\|u_0\|_{L^2}^3 +
\|F\|_{L_t^1L_x^2(J_T\times \bbr^n)}^3),
\end{align*}
where $C(N_1, N_2, N_3)$ is same as in Proposition
\ref{har-key-prop}.
 \end{lem}

We show this by using Proposition \ref{har-key-prop} which works when all $N_i>1$. However, due to $P_1(=\widetilde P_1)$ which has symbol supported in $B(0,2)$,  we need to use Corollary \ref{key-cor} when  one of $N_i$ is $1$. By Proposition \ref{har-key-prop} and Corollary \ref{key-cor}
we get
  \begin{eqnarray}\label{hart-keyy} \|\calh\|_{\amixed}\lesssim C(N_1, N_2,N_3) \normfg,\end{eqnarray}
for $N_1, N_2,N_3$ dyadic numbers $\ge 1$ and any admissible pair $(\widetilde q, \widetilde r)$.

\begin{proof}
   By taking $P_{N_j}$ to the equation and using Duhamel's formula, we have
  \begin{align*}
   \tilde u_j(t) = \schr (\tilde{u}_{0j} + \tilde{\mathbf{F}}_j(t)), \,j=1,2,3,
   \end{align*}
   where $\tilde u_{0j} = P_{N_j}(u_0(\cdot - \eta_j))$, $\tilde{\mathbf F}_j(t, x) = -iP_{N_j} (\int_0^t \tilde F_j(t') \,dt')$ and $\tilde F_j(t') = e^{-it'\Delta} (F(t', \cdot-\eta_j))$. Then we obtain
  $$
    \| |\nabla|^{2-n}(\tilde u_1 \tilde u_2) \tilde u_3\|_{\amixedI} \lesssim \sum_{k=1}^8 L_k,$$
  where
\begin{eqnarray*}
&L_1 \;= &\| \mathcal H (\tilde u_{01}, \tilde u_{02}, \tilde u_{03}) \|_{\amixedI}, \;\; L_2 \;= \;\| \mathcal H (\tilde{\mathbf F}_1, \tilde u_{02}, \tilde u_{03})\|_{\amixedI},\\
&L_3 \;= &\| \mathcal H (\tilde u_{01}, \tilde {\mathbf F}_2, \tilde u_{03})\|_{\amixedI},\;\;\; L_4 \;= \;\| \mathcal H (\tilde u_{01}, \tilde u_{02}, \tilde{\mathbf F}_3)\|_{\amixedI},\\
&L_5 \;= &\| \mathcal H (\tilde{\mathbf F}_1, \tilde{\mathbf F}_2, \tilde u_{03})\|_{\amixedI},\;\;\;\, L_6 \;= \;\| \mathcal H (\tilde u_{01}, \tilde{\mathbf F}_2, \tilde{\mathbf F}_3)\|_{\amixedI},\\
&L_7 \;= &\| \mathcal H (\tilde{\mathbf F}_1, \tilde u_{02},
\tilde{\mathbf F}_3)\|_{\amixedI},\;\; \;\,L_8 \;= \;\| \mathcal H (
\tilde{\mathbf F}_1, \tilde{\mathbf F}_2, \tilde{\mathbf
F}_3)\|_{\amixedI}.
\end{eqnarray*}
    The first term is easily handled by using \eqref{hart-keyy}.  We only consider $L_8$. The remaining cases are to be treated similarly. By Minkowski inequality we have
    \begin{align*}
      L_8 \lesssim & \left (\int_{J_T}  \left( \int_0^t\!\!\int_0^t\!\!\int_0^t \| \mathcal H  (\widetilde F_1(t_1), \widetilde F_2(t_2), \widetilde F_3(t_3))  \|_{L^{ {\tilde r}'}}\,dt_1dt_2dt_3\right)^{{\tilde q}'}  dt \right) ^{\frac{1}{ {\tilde q}'}}.
    \end{align*}
    This above is again bounded by
    \[ \int_{J_T\times J_T\times J_T}\|  \mathcal H  (\widetilde F_1(t_1), \widetilde F_2(t_2), \widetilde F_3(t_3))  \|_{\amixedI}\, dt_1 dt_2 dt_3. \]
    Applying \eqref{hart-keyy} again, we get
    $L_8  \lesssim C(N_1, N_2, N_3) \|F\|_{\mix{1}{2}(J_T\times \bbr^n)}^3$.
\end{proof}

 Let $\sigma(\xi)$ be infinitely differentiable so that for all
$\al \in N^{3n}$  and $\xi = (\xi_1,\xi_2,\xi_3) \in \bbr^{3n}$ there is a constant $c(\al)$ with
\begin{align}\label{sigma cond} |\pa_{\xi}^{\al} \sigma(\xi) | \le c(\al)(1+|\xi|)^{-|\al|}.\end{align}
Let us define a  multilinear operator $\Lambda$ by
\begin{align}\label{multi}
[\Lambda(f,g,h)](x) = \int
e^{ix\cdot\xi}\sigma(\xi_1,\xi_2,\xi_3)|\xi_2+\xi_3|^{-(n-2)}
\widehat{f}(\xi_1) \widehat{g}(\xi_2)\widehat{h}(\xi_3)\, d\xi_1
d\xi_2 d\xi_3,\end{align} where $\xi = \sum_{i = 1}^3 \xi_i$. We
notice that $\Lambda(f,g,h)=cf|\nabla|^{2-n}(gh)$ if $\sigma(\xi_1,\xi_2,\xi_3)=1$. Then we
have the following.

\begin{lem}\label{CMtri}
Let $\Lambda$ be defined as above and let
$u$ be a smooth solution of $iu_t + \Delta u = F$ on $J_T \times \bbr^n$ with the initial data $u_0$.  Then for  $(u_1, u_2, u_3)$ = $(P_{N_1}u, P_{N_2}u, P_{N_3}u)$ it holds that for any admissible pair $(\tilde q, \tilde r)$
  \[ \|\Lambda(u_1,u_2,u_3)\|_{\amixedI}\lesssim C(N_1, N_2,N_3)(\|u_0\|_{L^2}^3 + \|F\|_{L_t^1L_x^2(J_T\times \bbr^n)}^3),\]
where $C(N_1, N_2, N_3)$ is same as in Proposition \ref{har-key-prop}.
\end{lem}

\begin{proof}
We choose another Littlewood-Paley projections $\widetilde P_{N_i}, i = 1, 2, 3,$ such that $\widetilde P_{N_i}P_{N_i} = P_{N_i}$ and the corresponding cut-off multiplier $\widetilde \psi(N_i^{-1}\xi)$ is supported in $A(N_i)$. Then the multilinear operator $\La$ can be rewritten as
\begin{align*}
&[\Lambda(u_1,u_2,u_3)](x) \\
&= \int_{\bbr^{3n}} e^{ix\cdot\xi}\widetilde
\sigma_{1,2,3}(\xi_1,\xi_2,\xi_3)|\xi_2+\xi_3|^{-(n-2)}
\widehat{u_3}(\xi_1)
\widehat{u_2}(\xi_2)\widehat{u_1}(\xi_3)\, d\xi_1 d\xi_2 d\xi_3\\
&= c\int_{\eta =
(\eta_1,\eta_2,\eta_3)}\widehat{\widetilde \sigma_{1,2,3}}(\eta)
[u_1(\cdot-\eta_1)|\nabla|^{2-n}(u_2(\cdot-\eta_2)u_3(\cdot-\eta_3))](x)\,d\eta,
\end{align*}
where $\widetilde \sigma_{1,2,3}(\xi_1,\xi_2,\xi_3) = \widetilde
\psi(N_1^{-1}\xi_1)\widetilde \psi(N_2^{-1}\xi_2)\widetilde
\psi(N_3^{-1}\xi_3)\sigma(\xi_1,\xi_2,\xi_3)$. By the condition
\eqref{sigma cond},  support condition of $\widetilde \psi$ and
routine integration by parts,  we readily get a uniform bound
$\|\widehat{\widetilde \sigma_{1, 2, 3}}\|_{L^1(\bbr^{3n})} \le
C(\alpha)$ with respect to $N_1, N_2, N_3$ for sufficiently large
$|\alpha|$. By Minkowski's inequality and Lemma \ref{4.2} we
get
\begin{align*}
&\|\Lambda(u_1,u_2,u_3)\|_{\amixedI}\\
 &\lesssim \|\widehat{\widetilde \sigma_{1, 2, 3}}\|_{L^1(\bbr^{3n})} \sup_{\eta \in \bbr^{3n}} \|u_1(\cdot-\eta_1)|\nabla|^{2-n}(u_2(\cdot-\eta_2)u_3(\cdot-\eta_3))\|_{\amixedI}\\
&\lesssim C(N_2, N_3, N_1)(\|u_0\|_{L^2}^3 + \|F\|_{L_t^1L_x^2(J_T\times \bbr^n)}^3).
\end{align*}
Since $C(N_2, N_3, N_1) = C(N_1, N_2, N_3)$, we obtain the desired.
\end{proof}

The operator $I\lnabla $ behaves like $N^{1-s}|\xi|^s$ for $|\xi|\gtrsim N$. A Littlewood-Paley theory shows that the Leibniz rule holds for $I\lnabla (fg)$. Thus taking $I\lnabla$ to the equation \eqref{hartree},  we have
    \begin{align}\begin{aligned}\label{12}
    \| I\lnabla &[V(u)u]\|_{L^1_tL_x^2(J_T \times \bbr^n)}\\
    &\lesssim \left(\| u\|^2_{\mix{3}{\frac{6n}{3n-4}}(J_T\times \bbr^n)}\|I\lnabla u\|_{\mix{3}{\frac{6n}{3n-4}}(J_T\times \bbr^n)}\right.\\
    &\qquad\quad\left. + \| u\|^2_{\mix{\frac 83}{\frac{4n}{2n-3}}(J_T\times \bbr^n)}\|I\lnabla u\|_{\mix{4}{\frac{2n}{n-1}}(J_T\times \bbr^n)}\right) \lesssim  Z_I(T)^3.
 \end{aligned}\end{align}

\begin{lem}\label{multi-est}
 Let $u$ solve $iu_t + \Delta u =  V(u)u$ with the initial data $u_0 \in C_0^\infty$.
  Then for   $(u_1, u_2, u_3)$ = $(P_{N_1}u, P_{N_2}u, P_{N_3}u)$ it holds that for any $0 < s < 1$, $T > 0$ and admissible pair $(\tilde q, \tilde r)$
  \begin{align*}
  \| \Lambda(I\lnabla u_1, I\lnabla u_2, I\lnabla u_3)\|_{\amixedI} \lesssim C(N_1, N_2,N_3) (Z_I(T)^3 + Z_I(T)^9),\end{align*}
  where $C(N_1, N_2, N_3)$ is same as in Proposition \ref{har-key-prop}.
\end{lem}

Now we are ready to prove the propositions. We first give the proof of  Proposition \ref{ACL}.

\subsection{Proof of Proposition \ref{ACL}}

Differentiating the energy $E(Iu)(t)$ of $Iu$
 with respect to time, $
\frac{d}{dt}E(Iu)(t)  =
\,\text{Re} \int_{\bbr^n}\overline{\pa_t Iu}[V(Iu)Iu - I(V(u)u)] dx.
$
Thus we get
\begin{align*}
E(Iu(T))-E(Iu(0)) = \,\text{Re} \int_0^T \int_{\bbr^n}
\overline{\pa_t Iu}
\left[ V(Iu)Iu- I(V(u)u)\right]\ dx \, dt'.
\end{align*}
We apply the Parseval formula to the right hand side and use the equation \eqref{hartree}
to get
\begin{align}\label{E1}
|E(Iu(T))-E(Iu(0))| \lesssim E_a + E_b,
\end{align}
where
\begin{eqnarray}\label{Ea}
 E_{a} &\equiv& \bigg|
 \int_0^T\!\!\!\int_{\sum_{j=1}^4 \xi_j=0} \left(1-
\frac{m(\xi_2+\xi_3+\xi_4)}{m(\xi_2)m(\xi_3)m(\xi_4)}\right)\,  |\xi_2+\xi_3|^{-(n-2)}\\ \nonumber
& & \quad\qquad \times \ \widehat { \Delta \overline{Iu}}(\xi_1)\,
\widehat{\overline{Iu}}(\xi_2) \, \widehat{Iu}(\xi_3) \,
\widehat{Iu}(\xi_4)\ d\xi_1 \, d\xi_2 \, d\xi_3 \,
d\xi_4 \, dt'\, \bigg|,
\end{eqnarray}
\begin{eqnarray}\label{Eb}
E_{b} &\equiv& \bigg| \int_0^T\!\!\!\int_{\sum_{j=1}^4 \xi_j=0}
\left(1-
\frac{m(\xi_2+\xi_3+\xi_4)}{m(\xi_2)m(\xi_3)m(\xi_4) }\right)\,  |\xi_2+\xi_3|^{-(n-2)} \\ \nonumber
& & \quad \times \ \widehat{ (\overline{I(V(u)u}))}(\xi_1)\, \widehat{\overline{Iu}}(\xi_2) \,
\widehat{Iu}(\xi_3) \, \widehat{Iu}(\xi_4) \ d\xi_1 \, d\xi_2 \, d\xi_3
\, d\xi_4\, dt'\, \bigg|.
\end{eqnarray}
For both $E_a$ and $E_b$, we break $u$ into $u_i \equiv P_{N_i}u \;(i= 1,2,3,4)$ and exploit the interaction between Schr\"odinger waves of different frequency level using Proposition \ref{har-key-prop}.

For $E_a$ we show that for all $T> 0$ and $\ep > 0$
 \begin{equation}\label{Eannn}
 E_a \lesssim N^{-\frac 32 + \ep}(Z_I(T)^4 + Z_I(T)^{10}).
\end{equation}
It was shown  in \cite{ck} with $N^{-1+}$. The improvement  is actually due to the interaction gain of  $(N_j/N_k)^{\frac 12}$  in Lemma \ref{multi-est}. Let us set
\begin{equation}\label{B}
 B= B(N_2,N_3,N_4)\equiv \sup_{|\xi_2| \sim N_2, |\xi_3| \sim N_3, |\xi_4| \sim N_4}\left| 1- \frac{m(\xi_2+\xi_3+\xi_4)}{m(\xi_2)m(\xi_3)m(\xi_4)}\right| .
 \end{equation}
By  dyadic decomposition and factoring $B(N_2,N_3,N_4)$ out from the integral in $E_a$, we get
\begin{align*}
E_a
&\lesssim \sum_{N_1,N_2,N_3,N_4} B \left|\int_0^T\int_{\bbr^n} \mathcal{F}^{-1}[\Lambda(\Delta \overline{Iu_1}, \overline{Iu_2},Iu_3)](\xi_4)\mathcal{F}(Iu_4)(\xi_4)d\xi_4dt'\right| \nonumber\\
 &=\sum_{N_1,N_2,N_3,N_4} B \left|\int_0^T\int_{\bbr^n} [\Lambda(\Delta \overline{Iu_1}, \overline{Iu_2},Iu_3)](x)Iu_4(x)\,dx dt'\right|,
\end{align*}
where $\Lambda$ is the multiplier operator as defined by \eqref{multi} with the symbol
$$
\sigma(\xi_1, \xi_2, \xi_3) = \frac1{B(N_2,N_3,N_4)}\left(1 - \frac{m(\xi_2+\xi_3+\xi_4)}{m(\xi_2)m(\xi_3)m(\xi_4)}\right).
$$
Note that  $\sigma$ satisfies the condition \eqref{sigma cond}. Then for \eqref{Eannn} we need to show
\[\sum_{N_1,N_2,N_3,N_4} B\frac{N_1}{\langle N_2 \rangle \langle N_3 \rangle \langle N_4 \rangle} \mathcal E_a(\widetilde u_1, \widetilde u_2,\widetilde u_3,\widetilde u_4)\lesssim N^{-\frac 32 + \ep}(Z_I(T)^4 + Z_I(T)^{10}),\]
where $\widetilde u_1 = N_1^{-1}\Del\lnabla^{-1} u_1$ and $\widetilde u_j = N_j\lnabla^{-1}u_j, j = 2, 3, 4$, and
\[\mathcal E_a( w_1,  w_2, w_3, w_4)=\left|\int_0^t\int_{\bbr^n} [\Lambda(I \lnabla \overline{w_1}, I \lnabla \overline{w_2}, I \lnabla w_3)](x)I \lnabla w_4(x)\,dx dt'\right|.\]

Now note that $\mathcal E_a(\widetilde u_1, \widetilde u_2, \widetilde u_3, \widetilde u_4)=\mathcal E_a(\overline{\widetilde u_4}, \widetilde u_2, \widetilde u_3,\overline{\widetilde u_1}).
$
Hence by H\"older's inequality
\[\mathcal E_a(\widetilde u_1, \widetilde u_2, \widetilde u_3, \widetilde u_4)\le \| \Lambda(I\lnabla \widetilde  u_4, I\lnabla \widetilde u_2,I\lnabla \widetilde u_3)\|_{\amixedI} \|I \lnabla \overline{\widetilde u_1}\|_{\mix{\tilde q}{\tilde r}(J_T\times \bbr^n)}.\]
Taking admissible $(\tilde q, \tilde r)$ we now apply Corollary \ref{multi-est} and using  Sobolev imbedding (or Bernstein's inequality) and H\"ormander-Mikhlin theorem it follows that
\begin{equation}\label{e-a}
\mathcal E_a(\widetilde u_1, \widetilde u_2, \widetilde u_3, \widetilde u_4)\lesssim C(N_4,N_2,N_3)(Z_I(T)^4 + Z_I(T)^{10}).
\end{equation}
 For simplicity we also set
$$
a(N_1, N_2, N_3, N_4) \equiv \frac{B(N_2, N_3, N_4)N_1}{\langle N_2 \rangle \langle N_3 \rangle \langle N_4 \rangle} \times C(N_4,N_2,N_3). 
$$
Then for \eqref{Eannn} it is sufficient to show that for all $t \in J_T$ and $\ep > 0$
 \begin{equation}\label{Eann}
 \sum_{N_1,N_2,N_3,N_4\ge 1}  a(N_1,N_2,N_3,N_4) \lesssim N^{-\frac 32 + \ep}.
\end{equation}

\begin{proof}[Proof of \eqref{Eann}] Since $B$ and $C(N_2, N_3, N_4)$ are symmetric on the permutation of $N_2, N_3, N_4$, we may assume
\[N_2\geq N_3\geq N_4.\]  Then for the proof we  consider   the sums of  the three cases
$ N\gg N_2,\,\,  N_2 \gtrsim N \gg N_3 \geq N_4,\,\, N_3\gtrsim N$, separately.

\

\noindent\textbf{Case 1.} $N\gg N_2$. We have $m(\xi_i) =1$, $i=2,3,4,$ and $m(\xi_1) =1$ since $\sum_{i=1}^4 \xi_i =0$. So, the symbol $( 1- \frac{m(\xi_1)}{m(\xi_2)m(\xi_3)m(\xi_4)} ) = 0.$ Hence
\[ B(N_2, N_3, N_4)=0.\]

\noindent\textbf{Case 2.} $N_2 \gtrsim N \gg N_3 \geq N_4$. Since $\sum_{i=1}^4 \xi_i = 0$, we have $N_1 \sim N_2$. By the mean value theorem,
\begin{align*}
\Big|1- \frac{m(\xi_1)}{m(\xi_2)m(\xi_3)m(\xi_4)} \Big| &= \Big| \frac{m(\xi_2)-m(\xi_2+\xi_3+\xi_4)}{m(\xi_2)}  \Big|  \lesssim \frac{|\nabla m(\xi_2)\cdot (\xi_3+\xi_4)|}{m(\xi_2)} \lesssim \frac{N_3}{N_2}.
\end{align*}
Since $C(N_2, N_3, N_4) = (N_4/N_2)^\frac12$, we thus have \[a \lesssim \frac{1}{N_2N_4}\gain .\] Taking sum in the order of $N_2, N_3, N_4,$ we get
  \[ \sum_{N_1\sim N_2 \ge N\gg N_3\ge N_4} a \lesssim N^{-\frac 32}\ln N.\]

\noindent\textbf{Case 3.} $N_3\gtrsim N$. For this we need only to consider two subcases $N_1\sim N_2$ and $N_2 \gg N_1$ since the case $ N_1 \gg N_2$ cannot happen by the condition $\sum_i\xi_i = 0$.

\

\noindent\textbf{Case 3-1.} $N_3\gtrsim N$,  $N_1\sim N_2$. In this case, we have the bound
\[\Big|1- \frac{m(\xi_1)}{m(\xi_2)m(\xi_3)m(\xi_4)}  \Big|\frac{N_1}{N_2N_3N_4}\lesssim \frac{1}{N_3m(\xi_3)N_4m(\xi_4)}\]since $0 < m(\xi_i) \le 1$.
Then we consider  two possible cases $N_1\sim N_2 \ge N_3 \ge N_4 \gtrsim N$, $N_1\sim N_2 \ge N_3 \ge N \gg N_4$, separately. When $N_1\sim N_2 \ge N_3 \ge N_4 \gtrsim N$, we have
\[ a \lesssim \sup_{|\xi_3|\sim N_3, |\xi_4|\sim N_4} \frac{1}{N_3m(\xi_3)N_4m(\xi_4)}\gain
\lesssim \frac{1}{N^{2-2s}N_3^sN_4^s}\gain .\]
Hence, summation in the order of $N_2, N_3, N_4$ gives
\[ \sum_{N_1\sim N_2 \ge N_3 \ge N_4 \gtrsim N} a \lesssim N^{- 2}. \]
For the case $N_1\sim N_2 \ge N_3 \ge N \gg N_4$, it follows from the fact $m(\xi_4) = 1$ that
\[ a \lesssim \sup_{|\xi_3|\sim N_3, |\xi_4|\sim N_4} \frac1{N_3m(\xi_3)N_4} \gain  \lesssim \frac{1}{N^{1-s}N_3^sN_4}\gain .\]
Summing $N_2, N_3, N_4$, successively, we have acceptable bound
\[ \sum_{N_1\sim N_2 \ge N_3 \ge N_4 \gtrsim N} a \lesssim N^{-\frac  32}. \]

\noindent\textbf{Case 3-2.}  $N_3\gtrsim N$,  $N_2\gg N_1$. We have $N_2\sim N_3$ from $\sum_{i=1}^4 \xi_i =0$. Since $m(\xi_1) \ge m(\xi_2)$, we get
\begin{align}\label{bound}
\Big|1- \frac{m(\xi_2+\xi_3+\xi_4)}{m(\xi_2)m(\xi_3)m(\xi_4)} \Big|&\frac{N_1}{N_2N_3N_4}
\lesssim \frac{m(\xi_1)}{m(\xi_2)^2m(\xi_4)}\frac{N_1}{N_2^2N_4}.
\end{align}
We handle the cases $N_1\gtrsim N$ and $N_1 \le N$, separately.

If $N_1\gtrsim N$, we have three possible cases; $N_2\sim N_3\ge N_4 \ge N_1 \ge N $, $N_2\sim N_3\ge N_1 \ge N_4 \ge N  $,  $N_2\sim N_3\gg N_1 \gtrsim N \ge N_4 $.
When $N_2\sim N_3\ge N_4 \ge N_1 \ge N $,  using \eqref{bound} we get
\[ a \lesssim \frac{N_1^s}{N^{2-2s}N_2^{2s}N_4^s}\gain .\]
 Summing in the order of $N_2, N_4, N_1,$ we have
$\sum_{N_2\sim N_3\ge N_4 \ge N_1 \ge N} a \lesssim N^{-2}$, which is acceptable.
Similarly, when $N_2\sim N_3\ge N_1 \ge N_4 \ge N  $, it follows
\[ a \lesssim \frac{N_1^s}{N^{2-2s}N_2^{2s}N_4^s}\gain .\]
Then we get $ \sum_{N_2\sim N_3\ge N_1 \ge N_4 \ge N} a \lesssim N^{-2}$ by summing in $N_2, N_1, N_4$, successively.
Finally, when $N_2\sim N_3\gg N_1 \gtrsim N \ge N_4 $, we have
\[ a \lesssim \frac{N_1^s}{N^{1-s}N_2^{2s}N_4}\gain .\]
So summation  gives $\sum_{N_2\sim N_3\gg N_1 \gtrsim N \ge N_4 } a \lesssim N^{- \frac 32}.$

\smallskip

Now we turn to the case $N_1 \le N$. We again have three possible cases; $N_2\sim N_3\ge N_4 \ge N \ge N_1  $,  $N_2\sim N_3\ge N \ge N_4 \ge N_1  $,  $N_2\sim N_3\ge N \ge N_1 \ge N_4 $. Firstly, when $N_2\sim N_3\ge N_4 \ge N \ge N_1$, we have
\[  a \lesssim \frac{N_1}{N^{3-3s}N_2^{2s}N_4^s}\gain \]
which gives acceptable  bound $\sum_{N_2\sim N_3\ge N_4 \ge N \ge N_1} a \lesssim N^{-2} $.
For the case $N_2\sim N_3\ge N \ge N_4 \ge N_1  $ it follows that
\[ a \lesssim \frac{N_1}{N^{2-2s}N_2^{2s}N_4}\gain .\]
So, we have
$ \sum_{N_2\sim N_3\ge N \ge N_4 \ge N_1} a \lesssim N^{-2}.$
Finally when $N_2\sim N_3\ge N \ge N_1 \ge N_4   $,  we have
\[ a \lesssim \frac{N_1}{N^{2-2s}N_2^{2s}N_4}\gain .\]
Summation gives the bound $\sum_{N_2\sim N_3\ge N \ge N_1 \ge N_4 } a \lesssim N^{-\frac  32}$.
Thus we conclude the proof of \eqref{Eann}.
\end{proof}

Now we turn to   $E_b$ and claim
\begin{equation}\label{ebnn}
E_b\lesssim  N^{-\frac32 +\ep}(Z_I(T)^6 + Z_I(T)^{12}).
\end{equation}
Decomposing the integral in \eqref{Eb} dyadically,  factoring $B(N_2,N_3,N_4)$ out and using Plancherel's formula as before, we see that
\begin{align*}
E_b & \le
\sum_{N_1,N_2,N_3,N_4} B \left|\int_0^T\int_{\bbr^n} \mathcal{F}^{-1}[\Lambda( \overline{Iu_2},Iu_3,Iu_4)](\xi_1)\mathcal{F}(P_{N_1}I(V(u)\overline u))(\xi_1)\,d\xi_1dt'\right| \nonumber
\\
& = \sum_{N_1,N_2,N_3,N_4} B \left|\int_0^T\int_{\bbr^n} [\Lambda(I\overline{u_2},Iu_3,Iu_4)](x)P_{N_1}I(V(u)\overline{u})(x)\,dxdt'\right|.
\end{align*}
Hence, we see
\begin{equation}\label{ebb}
E_b \lesssim  \sum_{N_1,N_2,N_3,N_4}  \frac{B}{\langle N_2 \rangle \langle N_3 \rangle \langle N_4 \rangle}\times \mathcal E_b(u, \tilde u_2, \tilde u_3, \tilde u_4),
\end{equation}
where  $\mathcal E_b$ is defined by
\begin{align*}
\mathcal E_b(u, w_2, w_3, w_4)\equiv
\left|\int_0^T\int_{\bbr^n} [\Lambda(I \lnabla \overline{w_2},I \lnabla w_3,I \lnabla w_4)](x)P_{N_1}I(V(u)\overline{u})(x)\,dxdt'\right|.
\end{align*}
Here $\tilde u_j$ are defined by the same way as for $\mathcal E_a$.
We need  the following lemma to get a control of $\mathcal E_b$.

\begin{lem}\label{e-b} Let $u$ be a smooth solution of $iu_t = \Del u + V(u)u$ with initial data $u_0$ on $J_T \times \bbr^n$. Then there holds
\begin{align*}
\mathcal E_b \lesssim C(N_2,N_3,N_4)N_1(Z_I(T)^6 + Z_I(T)^{12}).
\end{align*}
\end{lem}

\begin{proof} For any admissible pair $(q, r)$, the H\"older's inequality  yields
\[\mathcal E_b \le \|\Lambda(I \lnabla \overline{\tilde u_2},I \lnabla \tilde u_3,I \lnabla \tilde u_4)\|_{L^{q'}_tL^{ r'}_x(J_T\times \mathbb R^n)}
\|P_{N_1}I(V(u)\overline u)\|_{L^{q}_tL^{ r}_x(J_T\times \mathbb R^n)} .\]
Applying Lemma \ref{multi-est} and H\"ormander-Mikhlin theorem we have
 \begin{align*}
\mathcal E_b \lesssim  C(N_2,N_3,N_4)(Z_I(T)^3 + Z_I(T)^{9})
\|P_{N_1}I(V(u)\overline u)\|_{L^{ q}_tL^{ r}_x(J_T\times \mathbb R^n)}.
\end{align*}
Then Lemma \ref{e-b} is the consequence of the estimate
\begin{align}\label{02}
\|P_{N_1}I(V(u)\overline u)\|_{L^{ q}_tL^{ r}_x(J_T\times \mathbb R^n)}
\lesssim  N_1(Z_I(T))^3
\end{align}
for admissible $(q, r)$ with $2 \le q \le 4$ if $n = 3$ and $2 \le q \le \infty$ if $n \ge 4$.

In fact, using Bernstein's inequality and H\"ormander-Mikhlin
theorem, we see that for $r\ge \tilde r$
\begin{align}\label{ber}
\|P_{N_1}I(V(u)\overline u)\|_{L^{r}_x}
\lesssim N_1^{\frac n{\tilde r}-\frac nr} \|P_{N_1}I(V(u)\overline u)\|_{L^{\tilde r}_x}
\lesssim N_1^{\frac n{\tilde r}-\frac nr-1} \|I \lnabla (V(u)\overline u)\|_{L^{\tilde r}_x}.
\end{align}
From Leibniz rule for the operator $I \lnabla$ and H\"older's
inequality with ${1} /{\tilde r}=1/{ r_1} +  {1}/{r_2} $ we bound
$\|I \lnabla (V(u)\overline u)\|_{L^{\tilde r}_x}$ by
 \begin{align*}
\| |x|^{-2}\ast (I \lnabla|u|^2)\|_{L^{r_1}_x} \|u\|_{L^{r_2}_x}+\||x|^{-2}\ast |u|^2\|_{L^{r_1}_x}\|I \lnabla u\|_{L^{r_2}_x}.
 \end{align*}
It follows from the fractional integration that $\|I \lnabla
(V(u)u)\|_{L^{\tilde r}_x}\lesssim \|I \lnabla u\|_{L^{r_2}_x}^3$
for $1/r_1 = 2/r_2 - (n-2)/n$ and $\frac{n}{n-1} < r_2 <\frac{2n}{n-2}$. Since $\tilde
r \ge 1$, the equation $1/\tilde r = 3/r_2 - 1 + 2/n$ also implies
$r_2 \ge 3n/(2n-2)$. Combining this with \eqref{ber} we get
\begin{align*}
\|P_{N_1}I(V(u)u)\|_{L_t^qL^{r}_x(J_T\times \bbr^n)}
\lesssim N_1^{n(\frac3{r_2}-\frac1r-1+\frac 2n)-1} \|I \lnabla u\|_{L_t^{3q}L^{r_2}_x(J_T\times \bbr^n)}^3.
\end{align*}
If $(q, r)$ and $(3q, r_2)$ are admissible, then $n(\frac3{r_2}-\frac1r-1+\frac 2n)-1 = 1$. The admissibility and range of $r_2$ ensure that $2 \le q \le 4$ if $n = 3$ and $2 \le q \le \infty$ if $n \ge 4$. This proves \eqref{02}.
\end{proof}

 Using \eqref{ebb},  Lemma \ref{e-b} and the H\"ormander-Mikhlin theorem, we have
\[
 E_b \lesssim   \sum_{N_1,N_2,N_3,N_4} B\frac{\langle N_1\rangle}{\langle N_2 \rangle \langle N_3 \rangle \langle N_4 \rangle}C(N_2,N_3,N_4)(Z_I(T)^6 + Z_I(T)^{12}).
 \]
Then from \eqref{Eann} we see $ \sum_{N_1,N_2,N_3,N_4} B\frac{\langle N_1\rangle}{\langle N_2 \rangle \langle N_3 \rangle \langle N_4 \rangle}C\lesssim N^{-\frac32+\epsilon}$.  Therefore we get the desired \eqref{ebnn}. This completes the proof of Proposition \ref{ACL}.


\subsection{Proof of Proposition \ref{ebound}}
We recall that $\mathcal{N}_{bad} = I(V(u)u)- V(Iu)Iu.$
Then by H\"{o}lder's inequality  we get
\begin{align*}
  \text{Error} &=  \left|\int_0^T\!\!\iint_{\bbr^n\times\bbr^n}|Iu(x,t)|^2\frac{y-x}{|y-x|}\cdot \big(\mathcal{N}_{bad}\nabla \overline{Iu} - Iu\nabla\overline{\mathcal{N}_{bad}}\big)(y,t) dxdydt\right| \\
  & \leq \left(\int_0^T\!\!\int_{\bbr^n}( |\mathcal{N}_{bad}|\cdot|\nabla Iu| + |\nabla\mathcal{N}_{bad}|\cdot|Iu|)\,dydt\right)\|Iu\|^2_{L_t^\infty L^2_x(J_T\times \bbr^n)}  \\
  &\lesssim  \|\langle\nabla\rangle\mathcal{N}_{bad}\|_{\amixed(J_T\times \bbr^n)}\|\langle\nabla\rangle Iu\|_{\tx{\tilde q}{\tilde r}}\|Iu\|^2_{L_t^\infty L^2_x(J_T\times \bbr^n)} \\
  &\lesssim \|\langle\nabla\rangle\big[I(V(u)u)- V(Iu)Iu \big]\|_{\amixedI}(Z_I(T))^3.
\end{align*}
Hence the proof of Proposition \ref{ebound} is reduced to showing that
\[
\|\langle\nabla\rangle\big[I(V(u)u)- V(Iu)Iu \big]\|_{\amixedI} \lesssim N^{-\frac 32}(Z_I(T))^3.
\]
For any fixed $\psi \in \amixed(J_T\times \bbr^n)$ we set $$E_c = \left|\int_0^T\!\!\int_{\bbr^n} \lnabla \big[I(V(u)u)- V(Iu)Iu  \big]\overline \psi \, dx dt\right|.$$
Then by duality  it suffices to show that for $\epsilon>0$
\begin{equation}\label{I-comm-est}
E_c \lesssim N^{-\frac 32+\ep}(Z_I(T)^3 + Z_I(T)^9)\| \psi\|_{\tx{\tilde q}{\tilde r}(J_T\times \bbr^n)}.
\end{equation}

We now follow the similar lines argument as  in the proof of the Proposition \ref{ACL}.
By Plancherel's theorem we have
\begin{align*}
E_c \sim \left|\int_0^T\!\! \int_{\sum_{j=1}^4 \xi_i = 0} \widetilde\sigma(\xi_2, \xi_3, \xi_4) |\xi_2+\xi_3|^{-(n-2)}
\widehat{\overline{Iu}}(\xi_2)\widehat{Iu}(\xi_3)\widehat{Iu}(\xi_4)\widehat{\overline \psi}(\xi_1)d\xi dt\right|,
\end{align*}
where  \[\widetilde \sigma(\xi_2, \xi_3, \xi_4) = \langle\xi_2+\xi_3+\xi_4\rangle\left(1-
\frac{m(\xi_2+\xi_3+\xi_4)}{m(\xi_2)m(\xi_3)m(\xi_4)}\right).\]
We decompose $u_1, u_2,u_3$ and $\psi$ into the sum of dyadic pieces $u_j = P_{N_j}u (j = 1, 2, 3, 4)$ and $\psi_1=P_{N_1}\psi$.
Let us define the maximum of $|\widetilde \sigma|$ on each dyadic piece by
\begin{equation*}
 \widetilde B= \widetilde B(N_2,N_3,N_4)\equiv \sup_{|\xi_2| \sim N_2, |\xi_3| \sim N_3, |\xi_4| \sim N_4}|\widetilde \sigma(\xi_2, \xi_3, \xi_4)| .
 \end{equation*}
We now set $\sigma(\xi_2,\xi_3,\xi_4)= \widetilde B^{-1}\widetilde \sigma(\xi_2, \xi_3, \xi_4)$
and define the multilinear operator $\La$ to be as in \eqref{multi} with the symbol $\sigma$.
Then $$E_c \lesssim \sum_{N_1,N_2, N_3, N_4}\left|\int_0^T\!\!\int_{\bbr^n} [\La(\overline{Iu_2}, Iu_3, Iu_4)](x)\overline \psi_1(x)\,dxdt\right|. $$
Hence, as before we see
\begin{equation}\label{ebbbb}
E_c \lesssim  \sum_{N_1,N_2,N_3,N_4}  \frac{\widetilde B}{\langle N_2 \rangle \langle N_3 \rangle \langle N_4 \rangle}\times \mathcal E_c(u, \tilde u_2, \tilde u_3, \tilde u_4),
\end{equation}
where $\widetilde u_j = \langle N_j \rangle\langle \nabla \rangle^{-1} u_j, j = 2, 3, 4$ and $\mathcal E_c$ is defined by
\[\mathcal E_c( \psi_1, w_2, w_3, w_4)=\left|\int_0^T\!\!\int_{\bbr^n} [\Lambda(I \lnabla \overline{w_2}, I \lnabla w_3, I \lnabla w_4)](x)\overline \psi_1(x)\,dx dt'\right|.\]
Using H\"older's inequality  and  Lemma \ref{multi-est} as before, we  get $\mathcal E_c\lesssim C(N_2, N_3, N_4)( Z_I(T)^3 + Z_I(T)^9)\| \psi\|_{\tx{\widetilde q}{\widetilde r}(J_T\times \bbr^n)}.$
Then by this and \eqref{ebbbb} the proof of \eqref{I-comm-est} is reduced to showing that for $\epsilon>0$
\[\sum_{N_1,N_2,N_3,N_4}  \frac{\widetilde B (N_2, N_3, N_4)}{\langle N_2 \rangle \langle N_3 \rangle \langle N_4 \rangle} \times C(N_2, N_3, N_4) \lesssim N^{-\frac 32+\ep}.\]
Finally notice that $\widetilde B \sim BN_1$, where $B$ is the same upper bound appearing in the estimates of $E_a$ and $E_b$. Then we get the desired bound from \eqref{Eann}. This completes the proof of Proposition \ref{ebound}.

\section*{Appendix}

\subsection*{Wave packet decomposition}
For a fixed ${\lambda}\gg 1$, let us define the spatial and
frequency  grids  $\mathcal Y$, $\mathcal V$, by
\[\mathcal Y={\lambda}^{1/2}\mathbb Z^{n}, ~\mathcal
V={\lambda}^{-1/2}\mathbb Z^{n}\cap Q(2),\]
respectively.
For each $(y,v)\in \mathcal Y\times \mathcal V$,     we associate a  tube $T_{y,v}$  given by
\[T_{y,v}=\{(x,t)\in\mathbb R^n\times \mathbb R:
|t|\le 4\lambda, ~ |x-(y+ 2tv) |\le \lambda^{1/2}\}.\]
Obviously $T_{y,v}$ meets $(y,0)$ and  its major direction is
parallel to $(2v,1)$. Let us  denote by $\mathcal T(\lambda)$ the
collection of these cubes.  Conversely  for a given  $T=T_{y,v}\in \mathcal T(\lambda)$,
we set
\[ y_T=y, ~~ v_T=v.\]

Let  $\eta$ be the function
satisfying $\supp\widehat\eta\subset Q(2)$ and $\sum_{k\in
\mathbb Z^n}\eta(\cdot-k)=1$. Let $\psi\in
C_0^\infty(B(0,1))$ with $\sum_{k\in \mathbb Z^n}\psi(\cdot-k)=1$.
For $T\in \mathcal T(\lambda)$  we also set
\[ f_T(x)=\eta(\frac{x-y_T}{\lambda^{1/2}})\mathcal F^{-1}[\widehat{f} \psi({\lambda}^{1/2}(\cdot-v_T))].\]
Then it is obvious that
\[ \schr f=\sum_{T\in \mathcal T(\lambda)} \schr f_T
\]
provided $\widehat f$ is supported in $Q(1)$.  Then by routine integration
by parts one can see that $\schr f_T$ is essentially supported in
$T$. More precisely, for any $\delta>0$ there is a $C=C(M, \delta)$ such that
\[
\label{decay}
|\schr f(x)|\le C\lambda^{-M}\|f\|_{L^2}  \text{ if } (x,t)\not\in \lambda^\delta T.
\]
For the detail of the wave packet decomposition see \cite{t5} (also see \cite{l}).
For the proof of Proposition \ref{interaction}, we use the following estimates due to Tao \cite{t5}.

\begin{lem}[Relation $\sim$ between wave packets and $b$] \label{wavepacket}
Let $1\ll \lambda$, $0<\delta\ll 1$ and $\{b\}$ be the collection
of the cubes $b$ of side length $\sim \lambda^{1-\delta}$ partitioning
$Q(\lambda)\times(-\lambda,\lambda)$.   Suppose that  $f, g\in L^2$ with
$\widehat f,\widehat g$ supported in $ Q(3/2)$ and they are
decomposed  at scale $\lambda$ such that
\[f=\sum_{T\in\mathcal T(\lambda)} f_T, \quad g=\sum_{T\in\mathcal T(\lambda)} g_T.\]
Then if $dist(\supp \widehat f, \supp \widehat g)\sim 1$, then
there is a relation $\sim $ between  tubes $T\in \mathcal T(\lambda)$ and cubes $b\in \{b\}$
 such that for any $\epsilon>0$,
\begin{equation}
\label{l2sum} \sum_{b} \|\sum_{T\sim b} f_T\|_{L^2}^2\le
C\lambda^\epsilon\|f\|_{L^2}^2,
 ~~\sum_{b} \|\sum_{T\sim b} g_T\|_{L^2}^2\le C\lambda^\epsilon\|g\|_{L^2}^2,
 \end{equation}
and for any $b$ and $\epsilon>0$,
\begin{equation}
\label{notb} \|\sum_{T\not \sim b \text{ or } T'\not \sim b}
e^{it\Delta}f_T e^{it\Delta}g_{T'}\|_{L^2(b)}\le
C\lambda^\epsilon\lambda^{c\delta-(n-1)/4} \|f\|_{L^2}\|g\|_{L^2}.
\end{equation}
with $c$ independent of $\delta,\epsilon$.
\end{lem}

\end{document}